\documentclass[reqno,12pt]{amsart}
\usepackage[margin=1in]{geometry}
\usepackage{amsfonts}
\usepackage[all]{xy}
\usepackage{enumitem}
\usepackage{mathtools}
\usepackage{amssymb,amsmath,url}
\usepackage{setspace}
\usepackage{amsthm}
\usepackage{cite}
\usepackage{color}
\numberwithin{equation}{section}
\usepackage[normalem]{ulem}
\usepackage{graphicx}
\usepackage{lmodern}
\graphicspath{ {images/} }
\usepackage{pdfpages}

\author{Zhenghui Huo, Nathan A. Wagner and Brett D. Wick}
\title[Bekoll\'e-Bonami estimates on some pseudoconvex domains]{Bekoll\'e-Bonami estimates on some pseudoconvex domains}
\begin{document}
	\thanks{BDW's research is partially supported by National Science Foundation grants DMS \# 1560955 and DMS \# 1800057. NAW's research is supported by National Science Foundation grant DGE \#1745038 }
\address{Zhenghui Huo, Department of Mathematics and Statistics, The University of Toledo,  Toledo, OH 43606-3390, USA}
\email{zhenghui.huo@utoledo.edu}
\address{Nathan A. Wagner, Department of Mathematics and Statistics, Washington University in St. Louis,  St. Louis, MO 63130-4899, USA}
\email{nathanawagner@wustl.edu}
\address{Brett D. Wick, Department of Mathematics and Statistics, Washington University in St. Louis,  St. Louis, MO 63130-4899, USA}
\email{bwick@wustl.edu}
		\newtheorem{thm}{Theorem}[section]
	\newtheorem{cl}[thm]{Claim}
	\newtheorem{lem}[thm]{Lemma}
		\newtheorem{ex}[thm]{Example}
	\newtheorem{de}[thm]{Definition}
	\newtheorem{co}[thm]{Corollary}
	\newtheorem*{thm*}{Theorem}
	\theoremstyle{definition}
		\newtheorem{rmk}[thm]{Remark}
		
	\maketitle
\begin{abstract}
We establish a weighted $L^p$ norm estimate for the Bergman projection for a class of pseudoconvex domains.  We obtain  an upper bound for the weighted $L^p$ norm when the domain is, for example, a bounded smooth strictly pseudoconvex domain, a pseudoconvex domain of finite type in $\mathbb C^2$, a convex domain of finite type in $\mathbb C^n$, or a decoupled domain of finite type in $\mathbb C^n$. The upper bound is related to the Bekoll\'e-Bonami constant and is sharp. When the domain is smooth, bounded, and strictly pseudoconvex, we also obtain a lower bound for the weighted norm. As an additional application of the method of proof, we obtain the result that the Bergman projection is weak-type $(1,1)$ on these domains.

\medskip

\noindent
{\bf AMS Classification Numbers}: 32A25, 32A36,  32A50, 42B20, 42B35

\medskip

\end{abstract}

\section{Introduction}
Let $\Omega\subseteq \mathbb C^n$ be a bounded domain. Let $dV$ denote the Lebesgue measure on $\mathbb C^n$. The Bergman projection $P$ is the orthogonal projection from $L^2(\Omega)$ onto the Bergman space $A^2(\Omega)$, the space of all square-integrable holomorphic functions.
Associated with $P$, there is a unique function $K_\Omega$ on $\Omega\times\Omega$ such that for any $f\in L^2(\Omega)$:
\begin{equation}
P(f)(z)=\int_{\Omega}K_\Omega(z;\bar w)f(w)dV(w).
\end{equation}
Let $P^+$ denote the positive Bergman operator defined by:
\begin{equation}
P^+(f)(z):=\int_{\Omega}|K_\Omega(z;\bar w)|f(w)dV(w).
\end{equation}
A question of importance in analytic function theory and harmonic analysis is to understand the boundedness of $P$ or $P^+$ on the space $L^p(\Omega, \sigma dV)$, where $\sigma$ is some non-negative locally integrable function on $\Omega$.
In \cite{BB78,Bekolle}, Bekoll\'e and Bonami established the following  for $P$ and $P^+$ on the unit ball $\mathbb B_n\subseteq \mathbb C^n$:
\begin{thm}[Bekoll\'e-Bonami]
	Let $T_z$ denote the Carleson tent over $z$ in $\mathbb B_n\in \mathbb C^n$ defined as below:
	\begin{itemize}
		\item $T_z:=\left\{w\in \mathbb B_n:\left|1-\bar w\frac{z}{|z|}\right|<1-|z|\right\}$ for $z\neq 0$, and
		\item $T_z:= \mathbb B_n$ for $z=0$.
	\end{itemize} 	Let the weight $\sigma$ be a positive, locally integrable function on $\mathbb B_n$. Let $1<p<\infty$. Then the following conditions are equivalent:
	\begin{enumerate}[label=\textnormal{(\arabic*)}]
		\item $P:L^p(\mathbb B_n,\sigma)\mapsto L^p(\mathbb B_n,\sigma)$ is bounded;
		\item $P^+:L^p(\mathbb B _n,\sigma)\mapsto L^p(\mathbb B_n,\sigma)$ is bounded;
		\item The Bekoll\'e-Bonami constant $\mathcal B_p(\sigma)$ is finite where: $$\mathcal B_p(\sigma):=\sup_{z\in \mathbb B_n}\frac{\int_{T_z}\sigma(w) dV(w)}{\int_{T_z}dV(w)}\left(\frac{\int_{T_z}\sigma^{-\frac{1}{p-1}} (w)dV(w)}{\int_{T_z}dV(w)}\right)^{p-1}.$$
	\end{enumerate} 
\end{thm}
In \cite{HWW}, we generalized  Bekoll\'e and Bonami's result to a wide class of pseudoconvex domains of finite type. To do so, we combined the methods of Bekolle \cite{Bekolle} with McNeal \cite{McNeal2003}.  This method of proof is qualitative, showing that the Bekoll\'e-Bonami class is sufficient for the weighted inequality of the projection to hold on those domains, and also necessary if the domain is strictly pseudoconvex.  However, the method of good-lambda inequalities in \cite{Bekolle} seems unlikely to give optimal estimates for the norm of the Bergman projection.  In this paper, we address the quantitative side of this question using sparse domination techniques.

Motivated by recent developments on the $A_2$-Conjecture by Hyt\"onen \cite{Hytonen} for singular integrals in the setting of Muckenhoupt weighted $L^p$ spaces, people have made progress on the dependence of the operator norm $\|P\|_{L^p(\mathbb B_n,\sigma)}$ on $\mathcal B_{p}(\sigma)$. In \cite{Pott}, Pott and Reguera gave a weighted $L^p$ estimate for the Bergman projection on the upper half plane. Their estimates are in terms of the Bekoll\'e-Bonami constant and the upper bound is sharp. Later, Rahm, Tchoundja, and Wick \cite{Rahm} generalized the results of Pott and Reguera to the unit ball case, and also obtained estimates for the Berezin transform. Weighted norm estimates of the Bergman projection have also been obtained \cite{ZhenghuiWick2} on the Hartogs triangle.

We use the known estimates of the Bergman kernel in \cite{Fefferman,Monvel,NRSW3,McNeal1,McNeal2, McNeal91} to establish the Bekoll\'e-Bonami type estimates for the Bergman projection on some classes of finite type domains. By finite type we mean that the D'Angelo 1-type \cite{D'Angelo82} is finite.  The domains of finite type we focus on are:
\begin{enumerate}
	\item domains of finite type in $\mathbb C^2$;
	\item strictly pseudoconvex domains with smooth boundary in $\mathbb C^n$; 
	\item convex domains of finite type in $\mathbb C^n$;
	\item decoupled domain of finite type in $\mathbb C^n$.
\end{enumerate}

Given functions of several variables $f$ and $g$, we use $f\lesssim g$ to denote that $f\leq Cg$ for a constant $C$. If $f\lesssim g$ and $g\lesssim f$, then we say $f$ is comparable to $g$ and write $f\approx g$. 

The main result obtained in this paper is:
\begin{thm}
	\label{t:main}
	Let $1<p<\infty$, and $p'$ denote the H\"older conjugate to $p$.
	Let $\sigma(z)$ be a positive, locally integrable function on $\Omega$. Set $\nu=\sigma^{-p^\prime/p}(z)$.
	Then the Bergman projection $P$  satisfies the following norm estimate on the weighted space $L^p(\Omega,\sigma)$:
	\begin{equation}\label{1.1}
\|P\|_{L^p(\Omega,\sigma)}\leq \|P^+\|_{L^p(\Omega,\sigma  )}\lesssim  \mathcal [\sigma]_p,
	\end{equation}
	where $[\sigma]_p=\left(\langle\sigma \rangle^{dV}_{\Omega}\left(\langle \nu\rangle^{dV}_{\Omega}\right)^{p-1}\right)^{1/p}+pp^\prime\left(\sup_{\epsilon_0>\delta>0, z\in \mathbf b\Omega}\langle\sigma \rangle^{dV}_{B^\#(z,\delta)}\left(\langle \nu\rangle^{dV}_{B^\#(z,\delta)}\right)^{p-1}\right)^{\max \{1,\frac{1}{p-1}\}}.$
\end{thm}	
The tent $B^\#(z,\delta)$ above is slightly different from the tent we use in \cite{HWW} in order to fit in the machinery of dyadic harmonic analysis. These two tents are essentially equivalent. The construction of $B^\#(z,\delta)$ uses the existence of the projection map onto $\mathbf b\Omega$ which is defined in a small tubular neighborhood of $\mathbf b\Omega$. Hence the restriction $\delta<\epsilon_0$ is needed to make sure that $B^\#(z,\delta)$ is inside the tubular neighborhood. See Lemma \ref{lem3.2} and Definition \ref{de3.3} in Section 3. For the detailed definition of the constant $[\sigma]_p$ and its connection with the Bekoll\'e-Bonami constant $\mathcal B_p(\sigma)$, see Definition \ref{de3.4} and Remark \ref{Re3.5}. We provide a sharp example for the upper bound above. See Section 6. 

Using the asymptotic expansion of the Bergman kernel on a strictly pseudoconvex domain \cite{Fefferman, Monvel}, we showed in \cite{HWW} that when $\Omega$ is smooth, bounded, and strictly pseudoconvex, the boundedness of the Bergman projection $P$ on the weighted space $L^p(\Omega,\sigma)$ implies that the weight $\sigma$ is in the $\mathcal B_p$ class. Here we also provide the corresponding quantitative result, giving a  lower bound of  the weighted norm of $P$:
\begin{thm}	\label{t:main1}
Let $\Omega$ be a smooth, bounded, strictly pseudoconvex domain. Let $1<p<\infty$, and $p'$ denote the H\"older conjugate to $p$. 
Let $\sigma(z)$ be a positive, locally integrable function on $\Omega$. Set $\nu=\sigma^{-p^\prime/p}(z)$. Suppose the projection $P$ is bounded on $L^p(\Omega,\sigma)$. Then we have
\begin{equation}\label{1.2}
\left(\sup_{\epsilon_0>\delta>0, z\in \mathbf b\Omega}\langle\sigma \rangle^{dV}_{B^\#(z,\delta)}\left(\langle \nu\rangle^{dV}_{B^\#(z,\delta)}\right)^{p-1}\right)^{\frac{1}{2p}}\lesssim\|P\|_{L^p(\Omega,\sigma)}.
\end{equation}
If $\Omega$ is also Reinhardt, then 
\begin{equation}\label{1.3}
\left(\mathcal B_p(\sigma)\right)^{\frac{1}{2p}}\lesssim\|P\|_{L^p(\Omega,\sigma)},
\end{equation}
where $\mathcal B_p(\sigma)=\max{\left\{\langle\sigma \rangle^{dV}_{\Omega}\left(\langle \nu\rangle^{dV}_{\Omega}\right)^{p-1}, \;\;\sup_{\epsilon_0>\delta>0, z\in \mathbf b\Omega}\langle\sigma \rangle^{dV}_{B^\#(z,\delta)}\left(\langle \nu\rangle^{dV}_{B^\#(z,\delta)}\right)^{p-1}\right\}}.$
\end{thm}
When $\Omega$ is the unit ball in $\mathbb C^n$, the estimate (\ref{1.3}) was obtained in \cite{Rahm}. When $\Omega$ is smooth, bounded, and strictly pseudoconvex, it was proven in \cite{HWW} that if  $P$ is bounded on $L^p(\Omega,\sigma)$, then the constant $\mathcal B_p(\sigma)$ is finite. It remains unclear to us that, for a general strictly pseudoconvex domain $\Omega$, how $\|P\|_{L^p(\Omega,\sigma)}$ dominates the constant $\langle\sigma \rangle^{dV}_{\Omega}\left(\langle \nu\rangle^{dV}_{\Omega}\right)^{p-1}$.

The approach we employ in this paper is similar to the ones in \cite{Pott} and \cite{Rahm}. To prove (\ref{1.1}), we show that $P$ and $P^+$ are controlled by a positive dyadic operator. Then an analysis of the weighted $L^p$ norm of the dyadic operator yields the desired estimate. The construction of the dyadic operator uses a doubling quasi-metric on the boundary $\mathbf b\Omega$ of the domain $\Omega$ and a result of Hyt\"onen and Kairema \cite{HK}. For the domains we consider, estimates of the Bergman kernel function in terms of those quasi-metrics are known so that a domination of the Bergman projection by the dyadic operator is possible. There are other domains where estimates for the Bergman kernel function are known. We just focus on the above four cases and do not attempt to obtain the most general result. 

The paper is organized as follows: In Section 2, we recall the definitions and known results concerning the non-isotropic metrics and balls on the boundary of the domain. In Section 3, we give the definition of tents and the dyadic tents structure based on the non-isotropic balls in Section 2. In Section 4, we recall known estimates for the Bergman kernel function,  and prove a pointwise domination of the (positive) Bergman kernel function by a positive dyadic kernel.  In Section 5, we prove Theorem \ref{t:main}. In Section 6, we provide a sharp example for the upper bound in Theorem \ref{t:main}. In Section 7, we prove Theorem \ref{t:main1}. In section 8, we provide an additional (unweighted) application of the pointwise dyadic domination to show the Bergman projection is weak-type $(1,1)$. We point out some directions for generalization in Section 9.
\section*{Acknowledgment}
B. D. Wick's research is partially supported by National Science Foundation grants DMS \# 1560955 and DMS \# 1800057.  N. A. Wagner's research is partially supported by National Science Foundation grant DGE \#1745038.
We would like thank John D'Angelo, Siqi Fu, Ming Xiao, and Yuan Yuan for their suggestions and comments. We would also like to acknowledge the work of Chun Gan, Bingyang Hu,
and Ilyas Khan, who independently obtained a result similar to main one in this paper
at around the same time (see \cite{Hu}).
\section{Non-isotropic balls on the boundary}
In this section, we recall various definitions of quasi-metrics and their associated balls on the boundary of $\Omega$. When $\Omega$ is of finite type in $\mathbb C^2$ or strictly pseudoconvex in $\mathbb C^n$, distances on the boundary can be well described using sub-Riemannian geometry. Properties and equivalence of these distances can be found in \cite{NSW1,NSW,Nagel,BaloghBonk}.  For discussions about the sub-Riemannian geometry, see for example \cite{Bellaiche,Gromov1,Montgomery}.  

For convex or decoupled domains of finite type in $\mathbb C^n$, the boundary geometry could be more complicated. We use quasi-metrics in \cite{McNeal2,McNeal3,McNeal91}. In fact, all four classes of domains we consider in this paper can be referred to as the so-called ``simple domains'' in \cite{McNeal2003}. There it has been shown that estimates for the kernel function on these domains fall into a unified framework. When $\Omega$ is of finite type in $\mathbb C^2$ or strictly pseudoconvex in $\mathbb C^n$, the boundary geometry of $\Omega$ is relatively straightforward. The quasi-metric induced by special coordinates systems in \cite{McNeal1} and \cite{McNeal2003}  is essentially the same as the sub-Riemannian metric.

It is worth mentioning that estimates expressed using some quasi-metrics for the Bergman kernel function are known in other settings. See for example \cite{CD,Koenig,Raich}.
\subsection{Balls on the boundaries of domains of finite type in $\mathbb C^2$ or strictly pseudoconvex domains in $\mathbb C^n$} Let $\Omega$ be a bounded domain in $\mathbb C^n$ with $C^\infty$-smooth boundary. A defining function $\rho$ of $\Omega$ is a real-valued $C^\infty$ function on $\mathbb C^n$ with the following properties:
\begin{enumerate}
	\item $\rho(z)<0$ for all $z\in \Omega$ and $\rho(z)>0$ for all $z\notin \Omega$.
	\item $\partial \rho(z)\neq 0$ when $\rho(z)=0$.
\end{enumerate}
Such a $\rho$ can be constructed, for instance, using the Euclidean distance between the point $z$ and $\mathbf b\Omega$, the boundary of $\Omega$. One can also normalize $\rho$ so that $|\partial \rho|=1$. Let $T(\mathbf b\Omega)$ denote the tangent bundle of $\mathbf b\Omega$ and $\mathbb CT(\mathbf b\Omega)=T(\mathbf b\Omega)\otimes\mathbb C$ its complexification. Let $T^{1,0}(\mathbf b\Omega)$ denote the subbundle of $\mathbb CT(\mathbf b\Omega)$ whose sections are linear combinations of ${\partial}/{\partial z_j}$, and $T^{0,1}(\mathbf b\Omega)$ be its complex conjugate bundle. Their sum $H\mathbf (\mathbf b\Omega):=T^{1,0}(\mathbf b\Omega)+T^{0,1}(\mathbf b\Omega)$ is a bundle of real codimension one in the complex tangent bundle $\mathbb CT(\mathbf b\Omega)$. Let $\langle\cdot,\cdot\rangle$ denote the contraction of the one forms and vector fields, and let $[\cdot,\cdot]$ denote the Lie bracket of two vector fields.
Let $\lambda$ denote the Levi form, the Hermitian form on $T^{1,0}(\mathbf b\Omega)$ defined by 
$$\lambda (L,\bar K)=\langle \frac{1}{2}(\partial-\bar \partial)\rho,[L,\bar K]\rangle \;\;\;\text{ for }\;\;L,K\in T^{1,0}(\mathbf b\Omega).$$ By the Cartan formula for the exterior derivative of a one form, one obtains
\begin{align*}
\lambda(L,\bar K)=\left\langle -d \left(\frac{1}{2}(\partial-\bar \partial)\rho\right),L\wedge\bar K\right\rangle=\langle\partial\bar \partial\rho, L\wedge\bar K\rangle.
\end{align*}
Hence, the Levi form is the complex Hessian $(\rho_{i\bar j})$ of $\rho$, restricted to $T^{1,0}(\mathbf b\Omega)$.

The domain $\Omega$ is called pseudoconvex (resp. strictly pseudoconvex) if $\lambda$ is positive semidefinite (resp. definite), i.e., the complex Hessian $(\rho_{i\bar j})$ is positive semidefinite (resp. definite) when restricted to $T^{1,0}(\mathbf b\Omega)$.

Given $L\in T^{1,0}(\mathbf b\Omega)$, we say the type of $L$ at a point $p\in \mathbf b\Omega$ is $k$ and write $\text{Type}_pL=k$ if 
$k$ is the smallest integer such that there is a iterated commutator 
$$[\dots[[L_1,L_2],L_3],\dots,L_k]=\Psi_k,$$ 
where each $L_j$ is either $L$ or $\bar L$ such that $\langle \Psi_k,(\partial-\bar \partial)\rho\rangle\neq 0$. 

When $\Omega\subseteq \mathbb C^2$, the subbundle $T^{1,0}(\mathbf b\Omega)$ has dimension one at each boundary point $p$ and $\text{Type}_pL$ defines the type of the point $p$: A point $q\in \mathbf b\Omega$ is of finite type $m$ in the sense of Kohn \cite{Kohn} if
$\text{Type}_pL=m$. A domain is of Kohn finite type $m$ if every point $q\in \mathbf b\Omega$ is of Kohn finite type at most $m$. In the $\mathbb C^2$ case, Kohn's type and D'Angelo's 1-type are equivalent. See \cite{D'Angelo} for the proof.

When $\Omega$ is strictly pseudoconvex, the Levi form $\lambda$ is positive definite. Thus for every $L\in T^{1,0}(\mathbf b\Omega)$ and $p \in \mathbf b\Omega$, one has that $\text{Type}_pL=2$.

Using the defining function $\rho$, a local basis of $H (\mathbf b\Omega)$ can be chosen as follows. Let $p\in \mathbf b\Omega$ be a boundary point. We may assume that, after a unitary rotation, ${\partial}\rho(p)= dz_n$. Then there is a neighborhood $U$ of $p$ such that $\frac{\partial\rho}{\partial z_n}\neq0$ on $U$. We define $n-1$ local tangent vector fields on $\mathbf b\Omega\cap U$:
\[L_j=\rho_{z_n}\frac{\partial}{\partial z_j}-{\rho_{z_j}}{}\frac{\partial}{\partial z_n}\;\;\;\;\; \;\;\;\;\; j=1,2,3\dots, n-1;\]
and their conjugates:
\[\bar L_j=\rho_{\bar z_n}\frac{\partial}{\partial \bar z_j}-{\rho_{\bar z_j}}\frac{\partial}{\partial \bar z_n}\;\; \;\;\;\;\;\;\;\;j=1,2,3\dots, n-1.\]
Then the $L_j$'s span $T^{1,0}(\mathbf b\Omega)$ and the $\bar L_j$'s span $T^{0,1}(\mathbf b\Omega)$.  
We set \[S=\sum_{j=1}^{n}{\rho_{\bar z_j}}\frac{\partial }{\partial z_j}\;\;\; \text{ and } \;\; \;T=S-\bar S.\]
Then the $L_j$'s, $\bar L_j$'s and $T$ span $\mathbb CT(\mathbf b\Omega)$. 
Let $X_j$, $X_{n-1+j}$ be real vector fields such that \[L_j=X_j-iX_{n-1+j}\] for $j=1,\dots, n-1$. Then $X_j$'s and $T$ span the real tangent space of  $\mathbf b\Omega$ near the point $p$.

For every $k$-tuple of integers $l^{(k)}=(l_1,\dots,l_k)$ with $k\geq 2$ and $l_j\in\{1,\dots, 2n-2\}$, we define $\lambda_{l^{(k)}}$ to be the smooth function such that
\[[X_{l_k},[X_{l_{k-1}},[\dots[X_{l_2},X_{l_1}]\dots]]]=\lambda_{l^{(k)}}T\;\;\text{ mod }\;X_1,\dots, X_{2n-2},\]
 and define $\Lambda_k$ to be the smooth function 
 \[\Lambda_{k}(q)=\left(\sum_{\text{all }l^{(k)}}|\lambda_{l^{(k)}}(q)|^2\right)^{{1}/{2}}.\]
For $q\in U$ and $\delta>0$, we set \begin{align}\label{2.2}\Lambda(q,\delta)=\sum_{j=2}^{m}\Lambda_j(q)\delta^j.\end{align}
In the $\mathbb C^2$ case, a point $q\in \mathbf b\Omega$ is of finite type $m$ if and only if $\Lambda_2(q)=\cdots=\Lambda_{m-1}(q)=0$ but $\Lambda_m(q)\neq 0$.  When $\Omega$ is strictly pseudoconvex in $\mathbb C^n$, the function $\Lambda_2\neq 0$ on $\mathbf b\Omega$.

 Though the function $\Lambda$ is locally defined, one can construct a global $\Lambda$ that is defined on $\mathbf b{\Omega}$ and is comparable to its every local piece. In the  finite type in $\mathbb C^2$ case and the strictly pseudoconvex case, a global construction can be realized without using partitions of unity. We explain this now. When $\Omega$ is strictly pseudoconvex, $\Lambda_2$ does not vanish on the boundary of $\Omega$, therefore we can simply set $\Lambda(q,\delta)=\delta^2$. When $\Omega$ is of finite type in $\mathbb C^2$, global tangent vector fields $L_1$ and $S$ can be chosen on $\mathbf b\Omega$:
 \begin{align*}
 &L_1={\rho_{z_2}}\frac{\partial}{\partial z_1}-{\rho_{z_1}}\frac{\partial}{\partial z_2},\\
 &S={\rho_{\bar z_1}}\frac{\partial}{\partial z_1}+{\rho_{\bar z_2}}\frac{\partial}{\partial z_2}.
 \end{align*}
Then the function $\Lambda$ induced by the above $L_1$ and $S$ is  a smooth function defined on $\mathbf b\Omega$. From now on,  we choose $\Lambda$ to be the smooth function induced by $L_1$ and $S$ on $\mathbf b\Omega$ when $\Omega$ is of finite type in $\mathbb C^2$, and choose $\Lambda(q,\delta)=\delta^2$ when $\Omega$ is strictly pseudoconvex in $\mathbb C^n$.

We recall several non-isotropic metrics on $\mathbf b\Omega$ that are locally equivalent:
\begin{de}
	For $p,q\in \mathbf b\Omega$, the metric $d_{1}(\cdot,\cdot)$ is defined by:
	\begin{align}
	d_{1}(p,q)=\inf\Big\{& \int_{0}^{1}|\alpha^\prime(t)|dt: \alpha \text{ is any piecewise smooth map from }[0,1] \text{ to } \mathbf b \Omega\nonumber\\&\text{ with } \alpha(0)=p, \alpha(1)=q, \text{ and } \alpha^\prime(t)\in H_{\alpha(t)}(\mathbf b\Omega)\Big\}.
	\end{align}
Equipped with the metric $d_1$, we define the ball $B_1$ centered at $p\in \mathbf b\Omega$ of radius $r$ to be \begin{align}B_{1}(p,r)=\{q\in \mathbf b\Omega:d_{1}(p,q)<r\}.\end{align}
\end{de}

 \begin{de}
 	For $p,q\in \mathbf b\Omega$, the metric $d_{2}(\cdot,\cdot)$ is defined by:
 	\begin{align}
 	d_{2}(p,q)=\inf\Big\{& \delta: \text{ There is a piecewise smooth map } \alpha \text{ from } [0,1] \text{ to } \mathbf b \Omega\nonumber\\&\text{ with } \alpha(0)=p, \alpha(1)=q, \alpha^\prime(t)=\sum_{j=1}^{2n-2}a_j(t)X_j(\alpha(t)),\text{ and } |a_j(t)|<\delta\Big\}.
 	\end{align}
 Equipped with the metric $d_2$, we define the ball $B_2$ centered at $p\in \mathbf b\Omega$ of radius $r$ to be \begin{align}B_{2}(p,r)=\{q\in \mathbf b\Omega:d_{2}(p,q)<r\}.\end{align}
 \end{de}
\begin{de}
	For $p,q\in \mathbf b\Omega$, the metric $d_{3}(\cdot,\cdot)$ is defined by:
	\begin{align}
	d_{3}(p,q)=\inf\Big\{& \delta: \text{ There is a piecewise smooth map } \alpha \text{ from } [0,1] \text{ to } \mathbf b \Omega \text{ with }\nonumber\\& \alpha(0)=p,  \alpha(1)=q, \text{ and } \alpha^\prime(t)=\sum_{j=1}^{2n-2}a_j(t)X_j(\alpha(t))+b(t)T(\alpha(t)),\nonumber\\&\text{where } |a_j(t)|<\delta, |b(t)|<\Lambda(p,\delta)\Big\}.
	\end{align}
Equipped with the metric $d_3$, we define the ball $B_3$ centered at $p\in \mathbf b\Omega$ of radius $r$ to be \begin{align}B_{3}(p,r)=\{q\in \mathbf b\Omega:d_{3}(p,q)<r\}.\end{align}
\end{de}
It is known that when the domain is strictly pseudoconvex in $\mathbb C^n$, or of finite type in $\mathbb C^2$, the quasi-metrics $d_1$, $d_2$, and $d_3$ are locally equivalent (cf. \cite{NSW1,NSW,Nagel}), i.e. there are positive constants $C_1$, $C_2$ and $\delta$ so that when $d_i(p,q)<\delta$ with $i\in\{1,2,3\}$, \[C_1d_j(p,q)<d_i(p,q)<C_2d_j(p,q) \;\;\;\text{ for } j\in\{1,2,3\}.\]
As a consequence, the balls $B_j$ are also equivalent in the sense that for small $\delta$, there are positive constants $C_1$, $C_2$ such that \[B_i(p,C_1\delta)\subseteq B_j(p,\delta)\subseteq B_i(p,C_2\delta) \;\;\;\text{ for } i,j\in\{1,2,3\}.\]

It is worth noting that the definition $d_1(\cdot,\cdot)$ does not depend on how we choose the local vector fields. Moreover,  if $d_1(p,q)<\delta$, then for some positive constants $C_1,C_2$, \begin{align}\label{2.81}C_1\Lambda(p,\delta)\leq\Lambda(q,\delta)\leq C_2\Lambda(p,\delta).\end{align}
To introduce the Ball-Box Theorem below, we also need to define balls using the exponential map.
\begin{de}
	For $q\in \mathbf b\Omega$ and $\delta>0$, set
	\[B_4(q,\delta)=\left\{p\in \mathbf b\Omega:p=\exp\left(\sum_{j=1}^{2n-2}a_jX_j(q)+bT(q)\right), \text{ where }|a_j|<\delta, \text{ and }|b|<\Lambda(p,\delta)\right\}.\]
\end{de}
\begin{thm}[Ball-Box Theorem] \label{thm 2.5}
There exist positive constants $C_1, C_2$ such that for any $q\in\mathbf b\Omega$ and any sufficiently small $\delta>0$, \[B_j(q,C_1\delta)\subseteq B_4(q,\delta)\subseteq B_j(q,C_2\delta) \;\;\text{ for } j\in\{1,2,3\} .\] 
\end{thm}
The proof of this theorem can be found in for example \cite{NSW,Bellaiche,Gromov1,Montgomery}. Variants of the Ball-Box Theorem also exist in the literature. The following version of the Ball-Box Theorem  is a consequence of Theorem \ref{thm 2.5} and can be found in \cite{BaloghBonk}.
\begin{co}[Ball-Box Theorem] \label{Cor2.7} Let $\Omega$ be a smooth, bounded, strictly pseudoconvex domain.
There exist positive constants $C_1$, $C_2$ such that for any $q\in\mathbf b\Omega$ and any sufficiently small $\delta>0$, \[\text{Box}(q,C_1\delta)\subseteq B_j(q,\delta)\subseteq \text{Box}(q,C_2\delta) \;\;\text{ for } j\in\{1,2,3,4\}.\] 
Here $\text{Box}(q,\delta)=\{q+Z_H+Z_N\in \mathbf b\Omega: |Z_H|<\delta, |Z_N|<\delta^2\}$ where $Z_{H}\in H_q(\mathbf b\Omega)$ and $Z_{N}$ is orthogonal to $H_q(\mathbf b\Omega)$.
\end{co}
We will only use this corollary for the strictly pseudoconvex case. See for example \cite{BaloghBonk}.

The next theorem provides estimates for the surface volume of $B_4$, and hence also for $B_j$ with $j=\{1,2,3\}$. See for example \cite{NSW}.
	\begin{lem}\label{lem 2.7}
		Let $\mu$ denote the Lebesgue surface measure on $\mathbf b\Omega$. Then there exist constants $C_1,C_2>0$ such that
		\begin{align}
	C_1 \delta^{2n-2}\Lambda(p,\delta)\leq\mu (B_{4}(p,\delta))\leq C_2 \delta^{2n-2}\Lambda(p,\delta).
		\end{align}
	\end{lem}
As a consequence of the definitions of $d_1$ and $\Lambda$ and Lemma \ref{lem 2.7}, we have the ``doubling measure property'' for the non-isotropic ball: There exists a positive constant $C$ such that for each $p\in \mathbf b\Omega$ and $\delta>0$,  
		\begin{align}\label{2.8}
\mu(B_{j}(p,\delta))\leq C\mu(B_{j}(p,\delta/2))\;\; \text{ for } j\in\{1,2,3,4\}.
	\end{align}

\subsection{Balls on the boundary of a convex/decoupled domain of finite type}
When $\Omega$ is a convex/decoupled domain of finite type in $\mathbb C^n$, non-isotropic sets can be constructed using a special  coordinate system of McNeal \cite{McNeal2,McNeal91,McNeal2003} near the boundary of $\Omega$. Let $p\in \mathbf b\Omega$ be a point of finite type $m$. For a small neighborhood $U$ of the point $p$, there exists a holomorphic coordinate system $z=(z_1,\dots,z_n)$ centered at a point $q\in U$ and defined on $U$ and quantities $\tau_1(q,\delta), \tau_2(q,\delta),\dots, \tau_n(q,\delta)$ such that 
\begin{align}\label{2.10}\;\;\;\;\;\;\tau_1(q,\delta)=\delta\;\;\; \text{ and }\;\;\;\delta^{1/2}\lesssim\tau_j(q,\delta)\lesssim\delta^{1/m}\;\;\text{ for }\;\; j=2,3,\dots,n.\end{align}
Moreover, the polydisc $D(q,\delta)$ defined by:
\begin{align}
D(q,\delta)=\{z\in\mathbb C^n:|z_j|<\tau_j(q,\delta),j=1,\dots,n\}
\end{align}
is the largest one centered at $q$ on which the defining function $\rho$ changes by no more than $\delta$ from its value at $q$, i.e. if $z\in D(q,\delta)$, then $|\rho(z)-\rho(q)|\lesssim \delta$.

The polydisc $D(q,\delta)$ is known to satisfy several ``covering properties'' \cite{McNeal3}: \begin{enumerate}\item There exists a constant $C>0$, such that for points $q_1,q_2\in U\cap \Omega$ with $D(q_1,\delta)\cap D(q_2,\delta)\neq \emptyset$, we have
\begin{align}\label{2.11}
 D(q_2,\delta)\subseteq CD(q_1,\delta) \text{ and }  D(q_1,\delta)\subseteq CD(q_2,\delta).
\end{align}
\item There exists a constant $c>0$ such that for $q\in U\cap \Omega$ and $\delta>0$, we have \begin{align}\label{2.12}D(q,2\delta)\subseteq cD(q,\delta).\end{align}\end{enumerate}
It was also shown in \cite{McNeal3} that $D(p,\delta)$ induces a global quasi-metric on $\Omega$. Here we will use it to define a quasi-metric on $\mathbf b\Omega$.

 For $q\in \mathbf b\Omega$ and $\delta>0$, we define the non-isotropic ball of radius $\delta$ to be the set
\[ B_5(q,\delta)={D(q,\delta)}\cap\mathbf b\Omega.\]
Set containments (\ref{2.11}), (\ref{2.12}), and the compactness and smoothness of $\mathbf b\Omega$  imply the following properties for $B_5$:
\begin{enumerate}\item There exists a constant $C$ such that for $q_1,q_2\in U\cap \mathbf b\Omega$ with $B_5(q_1,\delta)\cap B_5(q_2,\delta)\neq \emptyset$, \begin{align}\label{2.14}
 B_5(q_2,\delta)\subseteq CB_5(q_1,\delta) \text{ and }  B_5(q_1,\delta)\subseteq CB_5(q_2,\delta).
\end{align} \item There exists a constant $c>0$ such that for $q\in U\cap \Omega$ and $\delta>0$, we have \begin{align}\label{2.15}B_5(q,\delta)\subseteq cB_5(q,\delta/2)\;\;\;\;\;\text{ and }\;\;\;\;\;\mu(B_5(q,\delta))\approx \prod_{j=2}^{n}\tau_j^2(q,\delta).\end{align} \end{enumerate}

The balls $B_5$ induce a quasi-metric on $\mathbf b\Omega\cap U$. For $q,p\in \mathbf b \Omega\cap U$, we set
$\tilde d_5(q,p)=\inf\{\delta>0:p\in B_5(q,\delta)\}.$ To extend this quasi-metric $\tilde d_5(\cdot,\cdot)$ to a global quasi-metric $d_5(\cdot,\cdot)$ defined on $\mathbf b\Omega\times\mathbf b\Omega$, one can just patch the local metrics together in an appropriate way. The resulting quasi-metric is not continuous, but satisfies all the relevant properties. We refer the reader to \cite{McNeal3} for more details on this matter. Since $d_5(\cdot,\cdot)$ and $\tilde d_5(\cdot,\cdot)$ are equivalent, we may abuse the notation $B_5$ for the ball on the boundary induced by $d_5$. Then (\ref{2.14}) and (\ref{2.15}) still hold true for $B_5$.

\section{Tents and dyadic structures on $\Omega$}
From now on, the domain $\Omega$ will belong to one of the following cases:
\begin{itemize}
	\item a bounded, smooth, pseudoconvex domain of finite type in $\mathbb C^2$, \item a bounded, smooth, strictly pseudoconvex domain in $\mathbb C^n$, 
	\item a bounded, smooth, convex domain of finite type in $\mathbb C^n$, or
	\item a bounded, smooth, decoupled domain of finite type in $\mathbb C^n$.
\end{itemize}
Notations $d(\cdot,\cdot)$ and $B(p,\delta)$ will be used for \begin{itemize}
	\item the metric $d_1(\cdot,\cdot)$ and the ball $B_1(p,\delta)$ if $\Omega$ is pseudoconvex of finite type in $\mathbb C^2$ or strictly pseudoconvex in $\mathbb C^n$;
	\item  the metric $d_5(\cdot,\cdot)$ and the ball $B_5(p,\delta)$ if $\Omega$ is a convex/decoupled  domain of finite type.

\end{itemize} 
\begin{rmk}It is worth noting that even though we use the same notation $B(p,\delta)$ for balls on the boundary of $\Omega$, the constant $\delta$ has different geometric meanings in different settings. When $\Omega$ is a bounded, smooth, pseudoconvex domain of finite type in $\mathbb C^2$, or a bounded, smooth, strictly pseudoconvex domain in $\mathbb C^n$, $\delta$ represents the radius of the sub-Riemannian ball $B_1(p,\delta)$. When $\Omega$ is a bounded, smooth, convex (or decoupled) domain of finite type in $\mathbb C^n$, $2\delta$ is the height  in the $z_1$ coordinate of the polydisc $D(q,\delta)$ that defines $B_5(q,\delta)$. If $\Omega$ is the unit ball $\mathbb B_n$ which is  strictly pseudoconvex, convex, and decoupled, the ball $B_1(q,\delta)$ will be of similar size as the ball $B_5(q,\sqrt{\delta})$.\end{rmk}

\subsection{Dyadic tents on $\Omega$ and the $\mathcal B_p(\sigma)$ constant} The non-isotropic ball $B(p,\delta)$ on the boundary $\mathbf b\Omega$ induces ``tents'' in the domain $\Omega$.
To define what ``tents'' are we need the orthogonal projection map near the boundary. Let $\operatorname{dist}(\cdot,\cdot)$ denote the Euclidean distance in $\mathbb C^n$. For small $\epsilon>0$, set \begin{align*}&N_{\epsilon}(\mathbf b\Omega)=\{w\in \mathbb C^n: \operatorname{dist}(w,\mathbf b\Omega)<\epsilon\}.\end{align*}
\begin{lem}\label{lem3.2}
	For sufficiently small $\epsilon_0>0$, there exists a map $\pi:N_{\epsilon_0}(\mathbf b\Omega)\to \mathbf b\Omega$  such that
	\begin{enumerate}[label=\textnormal{(\arabic*)}]
		\item For each point $z\in N_{\epsilon_0}(\mathbf b\Omega)$ there exists a unique point $\pi(z)\in \mathbf b\Omega$ such that \[|z-\pi(z)|=\operatorname{dist}(z,\mathbf b\Omega).\]
		\item For $p\in \mathbf b\Omega$, the fiber $\pi^{-1}(p)=\{p-\epsilon n(p): -\epsilon_0\leq \epsilon<\epsilon_0\}$ where $n(p)$ is the outer unit normal vector of $\mathbf b\Omega$ at point $p$.
		\item The map $\pi$ is smooth on $N_{\epsilon_0}(\mathbf b\Omega)$.
		\item If the defining function $\rho$ is the signed distance function to the boundary, the gradient $\triangledown\rho$ satisfies \[\triangledown\rho(z)=n(\pi(z)) \;\text{ for } \;z\in N_{\epsilon_0}(\mathbf b\Omega).\]
	\end{enumerate}
\end{lem}
A proof of Lemma \ref{lem3.2} can be found in \cite{BaloghBonk}.
\begin{de}\label{de3.3}
	Let $\epsilon_0$ and $\pi$ be as in Lemma (\ref{lem3.2}). For $z\in \mathbf b\Omega$ and sufficiently small  $\delta>0$, the ``tent'' $B^\#(z,\delta)$ over the ball $B(z,\delta)$ is defined to be the subset of $N_{\epsilon_0}(\mathbf b\Omega)$ as follows:
	 When $\Omega$ is a pseudoconvex domain of finite type in $\mathbb C^2$ or a strictly pseudoconvex domain,
	\[B^\#(z,\delta)=B_1^\#(z,\delta)=\{w\in  \Omega: \pi(w)\in B_1(z,\delta), |\pi(w)-w|\leq \Lambda (\pi(w),\delta)\}.\] 
	When $\Omega$ is a convex (or decoupled) domain of finite type in $\mathbb C^n$,
		\[B^\#(z,\delta)=B_5^\#(z,\delta)=\{w\in  \Omega: \pi(w)\in B_5(z,\delta), |\pi(w)-w|\leq \delta\}.\] 
		For $\delta\gtrsim 1$ and any  $z\in \mathbf b\Omega$, we set $B^\#(z,\delta)=\Omega$.
\end{de}
For the ``tent'' $B^\#(z,\delta)$ to be within $N_{\epsilon_0}(\mathbf b\Omega)$, the constant  $\delta$ in Definition \ref{de3.3} needs to satisfy $\Lambda(z^\prime,\delta)<\epsilon_0$ for $z^\prime\in B_1(z,\delta)$ when $\Omega$ is of finite type in $\mathbb C^2$ or strictly pseudoconvex; and satisfy $\delta<\epsilon_0$ when $\Omega$ is a convex (or decoupled) domain in $\mathbb C^n$.

Given a subset $U\in \mathbb C^n$, let $V(U)$ denote the Lebesgue measure of $U$. By (\ref{2.81}) and the definitions of the tents $B^\#_1(z,\delta)$ and $B^\#_5(z,\delta)$, we have:
\begin{align}\label{3.1}
&V(B_1^\#(z,\delta))\approx \delta^{2n-2}\Lambda^2(z,\delta),\\\label{3.2}&V(B_5^\#(z,\delta))\approx \delta^{2}\prod_{j=2}^{n}\tau^2_j(z,\delta).
\end{align}
and hence also the ``doubling property'':
\begin{align}\label{3.3}
V(B^\#(z,\delta))\approx V(B^\#(z,\delta/2)).
\end{align} We give the definition of the Bekoll\'e-Bonami constant on $\Omega$.
For a weight $\sigma$ and a subset $U\subseteq \Omega$, we set $\sigma(U):=\int_U\sigma dV$ and let $\langle f\rangle^{\sigma dV}_U$ denote the average of the function $|f|$ with respect to the measure $\sigma dV$ on the set $U$:
\begin{equation*}
\langle f\rangle^{\sigma dV}_U=\frac{\int_{U}|f(w)|\sigma dV}{\sigma(U)}.
\end{equation*}
\begin{de}\label{de3.4}
Given  weights $\sigma(z)$ and $\nu=\sigma^{-p^\prime/p}(z)$ on $\Omega$,  the characteristic  $ [\sigma]_p$ of the weight $\sigma$ is defined by
\begin{equation}\label{3.40}
[\sigma]_p:=\left(\langle\sigma \rangle^{dV}_{\Omega}\left(\langle \nu\rangle^{dV}_{\Omega}\right)^{p-1}\right)^{1/p}+pp^\prime\left(\sup_{\epsilon_0>\delta>0, z\in \mathbf b\Omega}\langle\sigma \rangle^{dV}_{B^\#(z,\delta)}\left(\langle \nu\rangle^{dV}_{B^\#(z,\delta)}\right)^{p-1}\right)^{\max \{1,\frac{1}{p-1}\}}.
\end{equation} 
\end{de}
\begin{rmk}\label{Re3.5}
A natural generalization of the $\mathcal B_p$ constant in the above setting will be  \[\mathcal B_p(\sigma)=\max\left\{\langle\sigma \rangle^{dV}_{\Omega}\left(\langle \nu\rangle^{dV}_{\Omega}\right)^{p-1},\sup_{\epsilon_0>\delta>0, z\in \mathbf b\Omega}\langle\sigma \rangle^{dV}_{B^\#(z,\delta)}\left(\langle \nu\rangle^{dV}_{B^\#(z,\delta)}\right)^{p-1}\right\}.\]
It is not hard to see that $\mathcal B_p(\sigma)$ and $[\sigma]_p$ are qualitatively equivalent, i.e., $\mathcal B_p(\sigma)$ is finite if and only if $[\sigma]_p$ is finite. But they are not quantitatively equivalent. As one will see in the proof of Theorem \ref{t:main}, the products of averages of $\sigma$ and $\sigma^{1/(1-p)}$ over the whole domain and over the small tents will have different impacts on the estimate for the weighted norm of the projection $P$. The $\mathcal B_p(\sigma)$ above fails to reflect such a difference, and hence is unable to give the sharp upper bound. For the same reason, the claimed sharpness of the Bekoll\'e-Bonami bound in \cite{Rahm} is not quite correct. See Remark 6.1. This issue did not occur in the upper half plane case \cite{Pott} since the average over the whole upper half plane is not included in the $\mathcal B_p$ constant there. 
\end{rmk}

Now we are in the position of constructing dyadic systems on $\mathbf b\Omega$ and $\Omega$. 
Note that the ball $B(\cdot,\delta)$ on $\mathbf b\Omega$ satisfies the ``doubling property'' as in (\ref{2.8}) and (\ref{2.15}). By (\ref{2.81}) and (\ref{2.11}),  the surface area $\mu(B(q_1,\delta))\approx \mu(B(q_2,\delta))$ for any $q_1,q_2\in \mathbf b\Omega$ satisfying $d(q_1,q_2)\leq \delta$. Combining these facts yields that the metric $d(\cdot,\cdot)$ is a doubling metric, i.e. for every $q\in \mathbf b\Omega$ and $\delta>0$, the ball $B(q,\delta)$ can be covered by at most $M$ balls $B(x_i,\delta/2)$. Results of Hyt\"onen and Kairema in \cite{HK} then give the following lemmas:
\begin{lem}\label{lem3.5}Let $\delta$ be a positive constant that is sufficiently small and let $s>1$ be a parameter.
There exist reference points $\{p_j^{(k)}\}$ on the boundary $\mathbf b\Omega$ and  an associated collection of subsets $\mathcal Q=\{Q_j^{k}\}$ of $\mathbf b\Omega$ with $p_j^{(k)}\in Q_j^{k}$ such that the following properties hold:
\begin{enumerate}[label=\textnormal{(\arabic*)}]
	\item For each fixed $k$, $\{p_j^{(k)}\}$ is a largest set of points on $\mathbf b\Omega$ satisfying $d_1(p_j^{(k)},p_i^{(k)})> s^{-k}\delta$ for all $i,j$. In other words, if $p\in \mathbf b\Omega$ is a point that is not in $\{p_j^{(k)}\}$, then there exists an index $j_o$ such that $d_1(p,p_{j_o}^{(k)})\leq s^{-k}\delta$.
	\item For each fixed $k$, $\bigcup_j Q^k_j=\mathbf b\Omega$ and $Q^k_j\bigcap Q^k_i=\emptyset$ when $i\neq j$.
	\item For $k< l$ and any $i,j$, either $Q^k_j\supseteq Q^l_i$ or $Q^k_j\bigcap Q^l_i=\emptyset$.
	\item There exist positive constants $c$ and $C$ such that for all $j$ and $k$, \[B(p_j^{(k)},cs^{-k}\delta)\subseteq Q^k_j\subseteq B(p_j^{(k)},Cs^{-k}\delta).\]
	\item Each $Q_j^k$ contains of at most $N$ numbers of $Q^{k+1}_i$. Here $N$ does not depend on $k, j$.
	\end{enumerate}
\end{lem}

\begin{lem}\label{lem3.6}
Let $\delta$ and $\{p^{(k)}_j\}$ be as in Lemma \ref{lem3.5}. There are finitely many collections $\{\mathcal Q_l\}_{l=1}^{N}$ such that the following hold:\begin{enumerate}[label=\textnormal{(\arabic*)}]\item Each collection $\mathcal Q_l$ is associated to some dyadic points $\{z^{(k)}_j\}$ and they satisfy all the properties in Lemma \ref{lem3.5}.\item For any $z\in \mathbf b\Omega$ and small $r>0$, there exist $Q_{j_1}^{k_1}\in \mathcal Q_{l_1}$ and $Q_{j_2}^{k_2}\in \mathcal Q_{l_2}$ such that
\[Q_{j_1}^{k_1}\subseteq B(z,r)\subseteq Q_{j_2}^{k_2}\;\;\;\text{ and }\;\;\;\mu(B(z,r))\approx\mu(Q_{j_1}^{k_1})\approx\mu(Q_{j_2}^{k_2}).\]\end{enumerate}
\end{lem}
Setting the sets $Q_j^k$ in Lemma \ref{lem3.5} as the bases, we construct dyadic tents in $\Omega$ as follows:
\begin{de}\label{de3.7}
	Let $\delta$, $\{p^{(k)}_j\}$ and $\mathcal Q=\{Q_j^{k}\}$ be as in Lemma \ref{lem3.5}. We define the collection $ {\mathcal T}=\{\hat {K}_j^{k}\}$ of dyadic tents in the domain $\Omega$ as follows:
	\begin{itemize}
		\item When $\Omega$ is pseudoconvex of finite type in $\mathbb C^2$, or strictly pseudoconvex  in $\mathbb C^n$, we define
	\[\hat {K}_j^{k}:=\{z\in\Omega: \pi(z)\in Q_j^k \text{ and }|\pi(z)-z|<\Lambda(\pi(z),s^{-k}\delta)\}.\]
\item When $\Omega$ is a convex or decoupled domain of finite type in $\mathbb C^n$, we define
	\[\hat {K}_j^{k}:=\{z\in\Omega: \pi(z)\in Q_j^k \text{ and }|\pi(z)-z|<s^{-k}\delta\}.\]
\end{itemize}
\end{de}

\begin{lem}\label{lem3.8}Let $\mathcal T=\{\hat K^k_j\}$ be a collection of dyadic tents in Definition \ref{de3.7} and let $\{\mathcal Q_l\}_{l=1}^{N}$ be a collection of subsets in Lemma \ref{lem3.6}.  The following statements hold true:\begin{enumerate}[label=\textnormal{(\arabic*)}]\item For any $\hat {K}_j^{k}$, $\hat {K}_i^{k+1}$ in $\mathcal T$, either $\hat {K}_j^{k}\supseteq\hat {K}_i^{k+1}$ or $\hat {K}_j^{k}\bigcap\hat {K}_i^{k+1}=\emptyset$.\item For any $z\in \mathbf b\Omega$ and small $r>0$, there exist $Q_{j_1}^{k_1}\in \mathcal Q_{l_1}$ and $Q_{j_2}^{k_2}\in \mathcal Q_{l_2}$ such that
		\[\hat K_{j_1}^{k_1}\subseteq B^\#(z,r)\subseteq \hat K_{j_2}^{k_2}\;\;\;\text{ and }\;\;\;V(B^\#(z,r))\approx V(\hat K_{j_1}^{k_1})\approx V(\hat K_{j_2}^{k_2}).\]\end{enumerate}\end{lem}
\begin{proof}
Statement (1) is a consequence of the definition of $\hat K^k_j$ and  Lemma \ref{lem3.5}(3). Statement (2) is a consequence of the definitions of $B^\#(z,r)$, $\hat K^k_j$, and Lemma \ref{lem3.6}(2). 
\end{proof}
By Lemma \ref{lem3.8}(2), we can replace $B^\#(z,\delta)$ by $\hat K^k_{j}$ in the definition of $[\sigma]_p$ to obtain a quantity of comparable size:
\begin{equation}
[\sigma]_p\approx\left(\langle\sigma \rangle^{dV}_{\Omega}\left(\langle \nu\rangle^{dV}_{\Omega}\right)^{p-1}\right)^{1/p}+pp'\left(\sup_{1\leq l\leq N}\sup_{\hat K^k_j\in \mathcal T_{l}}\langle\sigma \rangle^{dV}_{\hat K^k_j}\left(\langle \nu\rangle^{dV}_{\hat K^k_j}\right)^{p-1}\right)^{\max \{1,1/(p-1)\}}.
\end{equation} 
From now on, we will abuse the notation $[\sigma]_p$ to represent both the supremum in $B^\#_{q}$ and the supremum in $\hat K^k_{j}$.
\subsection{Dyadic kubes on $\Omega$}  By choosing the parameter $s$ in Lemmas \ref{lem3.5} and \ref{lem3.6} to be sufficiently large, we can also assume that for any $p\in Q_i^{k+1}\subset Q_j^{k}$, one has \begin{align}
\label{3.4}
\Lambda(p,s^{-k-1}\delta)<\frac{1}{4}\Lambda(p,s^{-k}\delta)\end{align}\begin{de}\label{de3.9}For a collection $\mathcal T$ of dyadic tents, we define the center $\alpha_j^{(k)}$ of each tent $\hat {K}_j^{k}$ to be the point satisfying 
\begin{itemize}
	\item $\pi(\alpha_j^{(k)})=p^{(k)}_j$; and 
	\item $|p^{(k)}_j-\alpha_j^{(k)}|=\frac{1}{2}\sup_{\pi(p)=p^{(k)}_j}\operatorname{dist}(p,\mathbf b\Omega)$.
\end{itemize}
We set $K^{k}_{-1}=\Omega\backslash\left(\bigcup_{j}\hat  {K}_j^{0}\right)$, and for each point $\alpha_j^{(k)}$ or its corresponding tent $\hat K^k_j$, we define the dyadic ``kube''  
${K}_j^{k}:=\hat {K}_j^{k}\backslash\left(\bigcup_{l}\hat  {K}_l^{k+1}\right),$
where $l$ is any index with $p^{(k+1)}_l\in \hat{K}^{k}_j$. \end{de}
The following lemma for dyadic kubes holds true:\begin{lem}\label{3.10} Let $\mathcal T=\{\hat K^k_j\}$ be the system of tents induced by  $\mathcal Q$ in Definition \ref{de3.7}. Let $K^k_j$ be the kubes of $\hat K^k_j$. Then \begin{enumerate}[label=\textnormal{(\arabic*)}]\item $K^k_j$'s are pairwise disjoint and\; $\bigcup_{j,k}K^k_j=\Omega$.  \item When $\Omega$ is a finite type domain in $\mathbb C^2$ or a strictly pseudoconvex domain in $\mathbb C^n$, \begin{align}\label{3.5}V(K^k_j)\approx V(\hat K^k_j)\approx s^{-k(2n-2)}\delta^{2n-2}\Lambda(p^{(k)}_j,s^{-k}\delta).\end{align} When $\Omega$ is a convex or decoupled domain of finite type in $\mathbb C^n$, \begin{align}\label{3.6}V(K^k_j)\approx V(\hat K^k_j)\approx s^{-2k}\delta^{-2}\prod_{j=2}^{n}\tau_j^2(p^{(k)}_j,s^{-k}\delta).\end{align}\end{enumerate}\end{lem}
\begin{proof}
	Statement (1) holds true by the definition of $K^k_j$. The estimates for $V(\hat K^k_j)$ in (\ref{3.5}) and (\ref{3.6}) follow from (\ref{3.1}), (\ref{3.2}) and Lemma \ref{lem3.8}(2). When the domain is convex or decoupled of finite type in $\mathbb C^n$, the height of $\hat K_j^k$ is $s$ times the height of the tent $\hat K_j^k\backslash K_j^k$. Thus $V(\hat K_j^k)\approx V(\hat K_j^k\backslash K_j^k)$ which also implies $V(\hat K_j^k)\approx V(K_j^k)$. For the finite type in $\mathbb C^2$ case and strictly pseudoconvex case, it follows by (\ref{3.4}) that $\Lambda(p,s^{-k-1}\delta)<\frac{1}{4}\Lambda(p,s^{-k}\delta)$ for any $p\in Q_i^{k+1}\subset Q_j^{k}$ . Hence the height of $\hat K^k_j$ will be at least 4 times the height of $\hat K^k_j\backslash K^k_j$. Thus $V(\hat K_j^k)\approx V(K_j^k)$.
\end{proof}
\subsection{Weighted maximal operator based on dyadic tents}  \begin{de}\label{de3.12}Let $\sigma$ be a positive integrable function on $\Omega$. Let  $\mathcal T_l$ be a collection of dyadic tents as in Definition 3.6. The weighted maximal operator $\mathcal M_{\mathcal T_l,\sigma}$ is defined by
	\begin{equation}
	\mathcal M_{\mathcal T_l,\sigma}f(w):=\sup_{\hat K^k_j\in\mathcal T_l}\frac{1_{\hat K^k_{j}}(w)}{\sigma(\hat K^k_{j})}\int_{\hat K^k_{j}}|f(z)|\sigma(z)dV(z).
	\end{equation}
\end{de}
\begin{lem}\label{lem3.12}
$\mathcal M_{\mathcal T_l,\sigma}$ is bounded on $L^p(\Omega,\sigma)$ for $1<p\leq \infty$. Moreover \begin{align}\label{3.8}\|M_{\mathcal T_l,\sigma}\|_{L^p(\sigma)}\lesssim p/(p-1).\end{align}
\end{lem}
\begin{proof}
It's obvious that $\mathcal M_{\mathcal T_l,\sigma}$ is bounded on $L^\infty(\Omega,\sigma)$. We claim $\mathcal M_{\mathcal T_l,\sigma}$  is of weak-type $(1,1)$, i.e. for $f\in L^1(\Omega,\sigma)$, the following inequality holds for all $\lambda>0$:
\begin{align}\label{3.9}
\sigma(\{z\in\Omega:M_{\mathcal T_l,\sigma}(f)(z)>\lambda\})\lesssim \frac{\|f\|_{L^1(\Omega,\sigma)}}{\lambda}.
\end{align}
Then the Marcinkiewicz Interpolation Theorem implies the boundedness of  $\mathcal M_{\mathcal T_l,\sigma}$ on $L^p(\Omega,\sigma)$ for $1<p\leq \infty$, and inequality (\ref{3.8}) follows from a standard argument for the Hardy-Littlewood maximal operator.

 For a point $w\in \left\{z\in\Omega:\mathcal M_{\mathcal T_l,\sigma}f(z)>\lambda\right\}$, there exists a unique maximal tent $\hat K^k_j\in \mathcal T$ that contains $w$ and satisfies:
\begin{equation}
\frac{1_{\hat K^k_j}(w)}{\sigma(\hat K^k_{j})}\int_{\hat K^k_{j}}|f(z)|\sigma(z)dV(z)>\frac{\lambda}{2}.
\end{equation}
Let $\mathcal I_\lambda$ be the set of  all such maximal tents $\hat K^k_j$. The union of these maximal tents covers the set $\left\{z\in \Omega:\mathcal M_{\mathcal T_l,\sigma}f(z)>\lambda\right\}$. Since the tents $\hat K^k_j$ are maximal, they are also pairwise disjoint. Hence
\begin{equation*}
\sigma(\left\{z\in \Omega:\mathcal M_{\mathcal T_l,\sigma}f(z)>\lambda\right\})\leq \sum_{\hat K^k_j \in \mathcal I_\lambda}\sigma(\hat K^k_j)\leq \sum_{\hat K^k_j\in \mathcal I_\lambda}\frac{2}{\lambda}\int_{ \hat{K}^k_j}f(z)\sigma(z)dV(z)\leq\frac{2\|f\|_{L^1(\Omega,\sigma)}}{\lambda}.
\end{equation*}
Thus inequality (\ref{3.9}) holds and $\mathcal M_{\mathcal T_l,\sigma}$ is weak-type (1,1). 
\end{proof}
\section{Estimates for the  Bergman kernel function}
We recall known estimates for the Bergman kernel function, and their relation with the volume of the tents in the previous section. 
\subsection{Finite Type in $\mathbb C^2$ Case} In \cite{NRSW3}, the estimate of the Bergman kernel has been expressed in terms of $d_1$ and $\Lambda(p,\delta)$. Similar results were also obtained in \cite{McNeal1}.
\begin{thm}[{\hspace{1sp}\cite{NRSW3,McNeal1}}] \label{thm4.1} Let $\epsilon_0$ be the same as in Lemma \ref{lem3.2}. Then for points  $p,q\in N_{\epsilon_0}(\mathbf b\Omega)$, one has
	\begin{align} 
|K_{\Omega}(p;\bar q)|\lesssim d_1(p,q)^{-2}\Lambda(\pi(p),d_1(p,q))^{-2}.
	\end{align}
As a consequence, there is a constant $c$ such that $p,q\in B^\#_1(\pi(p),cd_1(p,q))$ and 
\begin{align}
|K_{\Omega}(p;\bar q)|\lesssim (V( B^\#_1(\pi(p),cd_1(p,q))))^{-1}.
\end{align}
\end{thm}

\subsection{Strictly Pseudoconvex Case} When $\Omega$ is bounded, strictly pseudoconvex with smooth boundary, the behavior of the Bergman kernel function is well understood. In \cite{Fefferman,Monvel}, asymptotic expansions of the kernel function were obtained on and off the diagonal.  To obtain the $L^p$ mapping property of the Bergman projection, a weaker estimate as in \cite{CuckovicMcNeal} would suffice. The proof of the following theorem can be found in \cite{McNeal2003}.
\begin{thm}[{\hspace{1sp}\cite{McNeal2003}}]\label{thm4.2}
	Let $\Omega$ be a smooth, bounded, strictly pseudoconvex domain in $\mathbb C^n$ with a defining function $\rho$. For each $p\in \mathbf b\Omega$, there exists a neighborhood $U$ of $p$, holomorphic coordinates $(\zeta_1,\dots,\zeta_n)$ and a constant $C>0$, such that for $p,q\in U\bigcap \Omega$,
	\begin{align}
	|K_{\Omega}(p;\bar q)|\leq C\left(|\rho(p)|+|\rho(q)|+|p_n-q_n|+\sum_{j=1}^{n-1}|p_k-q_k|^2\right)^{-n-1}.
	\end{align}
	Here $p=(p_1,\dots,p_n)$ is in $\zeta$-coordinates.
\end{thm} 
Up to a unitary rotation and a translation, we may assume that, under the original $z$-coordinates, $\partial \rho(p)=dz_n$ and $p=0$, then the holomorphic coordinates $(\zeta_1,\dots,\zeta_n)$ in Theorem \ref{thm4.2} can be expressed as the biholomorphic mapping $\Phi(z)=\zeta$ with 
\begin{align*}
&\zeta_1=z_1\\
&\vdots\\
&\zeta_{n-1}=z_{n-1}\\
&\zeta_n=z_n+\frac{1}{2}\sum_{k,l=1}^{n}\frac{\partial^2\rho}{\partial z_l\partial z_k}(p)z_kz_l.
\end{align*}
The next theorem relates the estimate in Theorem \ref{thm4.2} to the measure of the tents:
\begin{thm} \label{thm4.3}Let $p$, $q$, and $(p_1,\dots,p_n)$ be the same as in Theorem \ref{thm4.2} .
There exists a constant $r>0$ such that the tent $B_1^\#(p,r)$ contains points $p$ and $q$, and  \begin{align}\label{4.2}r^2\approx\max\left\{|\rho(p)|+|\rho(q)|,|p_n-q_n|+\sum_{j=1}^{n-1}|p_k-q_k|^2\right\}.\end{align}Moreover, $|K_{\Omega}(p;\bar q)|\lesssim \left(V(B_1^\#(p,r))\right)^{-1}$.
\end{thm}
\begin{proof}Note that $\Phi$ is biholomorphic and can be approximated by the identity map near $p$. For any points $w,\eta$ in the neighborhood $U$ of the point $p$ in Theorem \ref{thm4.2}, the distance $d_1(w,\eta)$ is about the same when computed in coordinates $(\zeta_1,\dots,\zeta_n)$. $\Phi$ is also measure preserving since the complex Jacobian determinant $J_{\mathbb C}\Phi=1$. Therefore we may assume that those results about $d_1$ and volumes of the tents in Sections 2 and 3 hold true  in $\zeta$-coordinates.
Then by (\ref{3.1}) and the strict pseudoconvexity ($\Lambda(\pi(p),\epsilon)=\epsilon^2$), the estimate $$|K_{\Omega}(p;\bar q)|\lesssim (V(B_1^\#(p,r)))^{-1}$$ holds true for  $r>0$ that satisfies (\ref{4.2}). Therefore it is enough to show the existence of such a constant $r$.  Set $r_1=\sqrt{|p_n-q_n|+\sum_{j=1}^{n-1}|p_k-q_k|^2}$ and $r_2=\sqrt{|\rho(p)|+|\rho(q)|}$. Note that $\partial/\partial\zeta_1,\dots,\partial/\partial\zeta_{n-1}$ are in $H_p(\mathbf b\Omega)$ and $\partial/\zeta_n$ is orthogonal to $H_p(\mathbf b\Omega)$. It follows from the fact $\Lambda(\pi(p),\epsilon)=\epsilon^2$ 
 and Corollary \ref{Cor2.7} that  there exists a constant  $C_1$ such that the boundary point $\pi(q)\in \text{Box}(\pi(p),C_1r_1)$. On the other hand,
$
|\rho(p)|+|\rho(q)|\approx \operatorname{dist}(p,\mathbf b\Omega)+\operatorname{dist}(q,\mathbf b\Omega).
$
 Therefore there exists a constant $C_2$ such that $\Lambda(\pi(p),C_2r_2)>|\rho(p)|+|\rho(q)|$. Set $r=\max\{C_1r_1,C_2r_2\}$. Then $B_1^\#(\pi(p),r)$ contains both points $p,q$ and inequality (\ref{4.2}) holds.
\end{proof}
\subsection{Convex/Decoupled Finite Type Case} When $\Omega$ is a smooth, bounded, convex (or decoupled) domain of finite type in $\mathbb C^n$, estimates of the Bergman kernel function on $\Omega$ were obtained in \cite{McNeal2,McNeal91,McNeal2003}. See also \cite{NPT} for a correction of a minor issue in \cite{McNeal2}.
\begin{thm}\label{thm4.4}
Let $\Omega$ be a smooth, bounded, convex (or decoupled) domain of finite type in $\mathbb C^n$. Let $p$ be a boundary point of $\Omega$. There exists a neighborhood $U$ of $p$ so that for all $q_1,q_2\in U\cap \Omega$,
\begin{align}
|K_\Omega(q_1;\bar q_2)|\lesssim \delta^{-2}\prod_{j=2}^{n}\tau_j(q_1,\delta)^{-2},
\end{align}
where $\delta=|\rho(q_1)|+|\rho(q_2)|+\inf\{\epsilon>0:q_2\in D(q_1,\epsilon)\}$.
\end{thm}
 We can reformulate Theorem \ref{thm4.4} as below.
\begin{thm}\label{thm4.5}
	Let $\Omega$ be a smooth, bounded, convex (or decoupled)  domain of finite type in $\mathbb C^n$. Let $p$ be a boundary point of $\Omega$. There exists a neighborhood $U$ of $p$ so that for all $q_1,q_2\in U\cap \Omega$,
	\begin{align}\label{4.6}
	|K_\Omega(q_1;\bar q_2)|\lesssim \left(V(B_5^\#(\pi(q_1),\delta))\right)^{-1},
	\end{align}
	where $\delta=|\rho(q_1)|+|\rho(q_2)|+\inf\{\epsilon>0:q_2\in D(q_1,\epsilon)\}$. Moreover, there exists a constant $c$ such that $q_1, q_2\in B_5^\#(\pi(q_1),c\delta)$.
	\end{thm}
Here the
estimate (\ref{4.6}) follows from (\ref{3.2}). Recall that the polydisc $D(q,\delta)$ induces a global quasi-metric \cite{McNeal3} on $\Omega$. Then a triangle inequality argument using this quasi-metric yields the containment  $q_2\in B_5^\#(\pi(q_1),c\delta)$.
\subsection{Dyadic Operator Domination}
\begin{thm} \label{thm4.6} Let $\hat K_j^k$, $K_j^k$ be the tents and kubes with respect to $d$ and $B^\#$. Let $\{\mathcal T_l\}_{l=1}^N$ be the finite collections of tents induced by $\{\mathcal Q_l\}_{l=1}^N$ in Lemma \ref{lem3.6}. Then for $p,q\in \Omega$,
\begin{align}\label{4.7}
|K_{\Omega}(p;\bar q)|\lesssim (V(\Omega))^{-1}1_{\Omega\times{\Omega}}(p,q)+ \sum_{l=1}^{N}\sum_{\hat K_j^k\in \mathcal T_l}(V(\hat K_j^k))^{-1}1_{\hat K_j^k\times{\hat K_j^k}}(p,q).
\end{align}
\end{thm}
\begin{proof}It suffices to show that for every $p,q$, there exists a $\hat K^k_j\in \mathcal T_l$ for some $l$ such that $$|K_{\Omega}(p;\bar q)|\lesssim (V(\hat K_j^k))^{-1}1_{\hat K_j^k\times{\hat K_j^k}}(p,q).$$
	
When $\operatorname{dist}(p,q)\approx1$ or $\operatorname{dist}(p,\mathbf b\Omega)+\operatorname{dist}(q,\mathbf b\Omega)\approx 1$,  the pair $(p,q)$  is away from the boundary diagonal of $\Omega\times \Omega$. By Kerzman's Theorem \cite{Kerzman,Boas}, we have $$|K_{\Omega}(p;\bar q)|\lesssim 1\approx (V(\Omega))^{-1}\approx (V(\Omega))^{-1}1_{\Omega\times{\Omega}}(p,q).$$ 

We turn to the case when $\operatorname{dist}(p,q)$ and $\operatorname{dist}(p,\mathbf b\Omega)+\operatorname{dist}(q,\mathbf b\Omega)$ are both small and we may assume that both $p,q\in \Omega\cap N_{\epsilon_0}(\mathbf b\Omega)$. By Theorems \ref{thm4.1}, \ref{thm4.3}, and \ref{thm4.5}, there exists a small constant $r>0$ such that $p,q\in B^\#(\pi(p),r)$ and $$|K_{\Omega}(p;\bar q)|\lesssim (V(B^\#(\pi(p),r)))^{-1}.$$ By Lemma \ref{lem3.8}, there exists a tent $\hat K^k_j\in \mathcal T_l$ for some $l$ such that $B^\#(\pi(p),r)\subseteq \hat K^k_j$ and $$V(\hat K^k_j)\approx V(B^\#(\pi(p),r)).$$ Thus $p,q\in \hat K^k_j$ and $|K_{\Omega}(p;\bar q)|\lesssim (V(\hat K^k_j))^{-1}1_{\hat K_j^k\times{\hat K_j^k}}(p,q).$
\end{proof}
\section{Proof of  Theorem \ref{t:main}}
Given a function $h$ on $\Omega$, we set $M_h$ to be the multiplication operator by $h$:
\[M_h(f)(z):=h(z)f(z).\]
Let $\sigma$ be a weight on $\Omega$. Set $\nu(z):=\sigma^{{-p^\prime}/{p}}(z)$ where $p^\prime$ is the H\"older conjugate index of $p$. Then it follows that the operator norms of $P$ and $P^+$ on the weighted space $L^p(\Omega,\sigma)$ satisfy:
\begin{align}\label{5.1}
&\|P:L^p(\Omega, \sigma)\to L^p(\Omega, \sigma) \|=\|PM_\nu:L^p(\Omega,\nu)\to L^p(\Omega,\sigma)\|;\\&
\label{2.191}
\|P^+:L^p(\Omega, \sigma)\to L^p(\Omega, \sigma) \|=\|P^+M_\nu:L^p(\Omega,\nu)\to L^p(\Omega,\sigma)\|.
\end{align}
It suffices to prove the inequality for $\|P^+M_\nu:L^p(\Omega,\nu)\to L^p(\Omega,\sigma)\|$.

Let $\{\mathcal T_l\}_{l=1}^N$ be the finite collections of tents in Theorem \ref{thm4.6}. Then inequality (\ref{4.7}) holds: for $p,q\in \Omega$,
\begin{align}
|K_{\Omega}(p,\bar q)|\lesssim (V(\Omega))^{-1}1_{\Omega\times{\Omega}}(p,q)+\sum_{l=1}^{N}\sum_{\hat K_j^k\in \mathcal T_l}(V(\hat K_j^k))^{-1}1_{\hat K_j^k\times{\hat K_j^k}}(p,q).
\end{align}
Applying this inequality to the operator $P^+M_\nu$ yields
\begin{align}
\left|P^+M_\nu f(z)\right|=&\int_{\Omega}|K_{\Omega}(z;\bar w)\nu(w)f(w)|dV(w)
\nonumber\\\lesssim&\langle f\nu\rangle^{dV}_{\Omega}+\int_{\Omega}\sum_{l=1}^{N}\sum_{\hat K^k_j\in \mathcal T_l}\frac{1_{\hat K^k_{j}}(z)1_{\hat K^k_{j}}(w)\left|\nu(w) f(w)\right|}{V(\hat K^k_{j})}dV(w)\nonumber\\=&\langle f\nu\rangle^{dV}_{\Omega}+\sum_{l=1}^{N}\sum_{\hat K^k_j\in \mathcal T_l}{1_{\hat K^k_{j}}(z)}\langle  f\nu\rangle^{dV}_{\hat K^k_{j}}.
\end{align}
Set $Q^+_{0,\nu}(f)(z):=\langle f\nu\rangle^{dV}_{\Omega}$ and $Q^+_{l,\nu}f(z):=\sum_{\hat K^k_j\in \mathcal T_l}{1_{\hat K^k_{j}}(z)}\langle  f\nu\rangle^{dV}_{\hat K^k_{j}}$. Then it suffices to estimate the norm for $Q^+_{l,\nu}$ with $l=0,1,\dots, N$. The proof given below uses the argument for the upper bound of sparse operators in weighted theory of harmonic analysis, see for example \cite{Moen2012} and \cite{Lacey2017}. An estimate for the norm of $Q^+_{0,\nu}$ is easy to obtain by H\"older's inequality:
\begin{align}\label{5.50}
\frac{\|Q^+_{0,\nu}(f)\|^p_{L^p(\Omega,\sigma)}}{\|f\|^p_{L^p(\Omega,\nu)}}\lesssim \frac{(\langle f\nu\rangle^{dV}_{\Omega})^p\langle\sigma\rangle^{dV}_{\Omega}}{\int_\Omega|f|^p\nu dV}\lesssim \langle\sigma\rangle^{dV}_{\Omega}(\langle\nu\rangle^{dV}_{\Omega})^{p-1}.
\end{align}
Now we turn to $Q^+_{l,\nu}$ for $l\neq0$. 
Assume $p>2$. For any $g\in L^{p^\prime}(\Omega,\sigma)$,
\begin{align}\label{5.5}
\left|\left\langle Q^+_{l,\nu} f(z), g(z)\sigma(z)\right\rangle\right|=&\left|\int_{\Omega} Q^+_{l,\nu} f(z)g(z)\sigma(z)dV(z)\right|\nonumber\\=&\left|\int_{\Omega}\sum_{\hat K^k_j\in \mathcal T_l}{1_{\hat K^k_{j}}(z)}\langle  f\nu\rangle^{dV}_{\hat K^k_{j}}g(z)\sigma(z) dV(z)\right|\nonumber\\\leq &\sum_{\hat K^k_j\in \mathcal T_l}\langle  f\nu\rangle^{dV}_{\hat K^k_{j}}\int_{\hat K^k_{j}}|g(z)|\sigma (z) dV(z)\nonumber\\=&\sum_{\hat K^k_j\in \mathcal T_l}\langle  f\nu\rangle^{dV}_{\hat K^k_{j}}\langle  g\sigma\rangle^{dV}_{\hat K^k_{j}}V(\hat K^k_{j}).
\end{align}
Since $\langle  f\nu\rangle^{dV}_{\hat K^k_{j}}\langle  g\sigma\rangle^{dV}_{\hat K^k_{j}}=\langle  f\rangle^{\nu dV}_{\hat K^k_{j}}\langle\nu\rangle^{dV}_{\hat K^k_{j}}\langle g\rangle^{\sigma dV}_{\hat K^k_{j}} \langle \sigma\rangle^{dV}_{\hat K^k_{j}}$, it follows that
\begin{align}\label{5.51}
\sum_{\hat K^k_j\in \mathcal T_l}\langle  f\nu\rangle^{dV}_{\hat K^k_{j}}\langle  g\sigma\rangle^{dV}_{\hat K^k_{j}}V(\hat K^k_{j})=&\sum_{\hat K^k_j\in \mathcal T_{l}}\langle  f\rangle^{\nu dV}_{\hat K^k_{j}}\langle\nu\rangle^{dV}_{\hat K^k_{j}}\langle g\rangle^{\sigma dV}_{\hat K^k_{j}} \langle \sigma\rangle^{dV}_{\hat K^k_{j}}V(\hat K^k_{j})\nonumber\\=&\sum_{\hat K^k_j\in \mathcal T_{l}}\left(\langle\nu \rangle^{dV}_{\hat K^k_{j}}\right)^{p-1}  \langle \sigma\rangle^{dV}_{\hat K^k_{j}}\langle  f\rangle^{\nu dV}_{\hat K^k_{j}}\langle g\rangle^{\sigma dV}_{\hat K^k_{j}}V(\hat K^k_{j})\left(\langle\nu\rangle^{dV}_{\hat K^k_{j}}\right)^{2-p}.
\end{align}
 Then \begin{align}\label{5.6}
&\sum_{\hat K^k_j\in \mathcal T_{l}}\left(\langle\nu \rangle^{dV}_{\hat K^k_{j}}\right)^{p-1}  \langle \sigma\rangle^{dV}_{\hat K^k_{j}}\langle  f\rangle^{\nu dV}_{\hat K^k_{j}}\langle g\rangle^{\sigma dV}_{\hat K^k_{j}}V(\hat K^k_{j})\left(\langle\nu\rangle^{dV}_{\hat K^k_{j}}\right)^{2-p}\nonumber\\\leq&\sup_{1\leq l\leq N}\sup_{\hat K^k_j\in \mathcal T_{l}}\langle\sigma \rangle^{dV}_{\hat K^k_j}\left(\langle \nu\rangle^{dV}_{\hat K^k_j}\right)^{p-1}\sum_{\hat K^k_j\in \mathcal T_{l}}\langle  f\rangle^{\nu dV}_{\hat K^k_{j}}\langle g\rangle^{\sigma dV}_{\hat K^k_{j}}\left(V(\hat K^k_{j})\right)^{p-1}\left(\nu(\hat K^k_{j})\right)^{2-p}.
\end{align}
By Lemma \ref{3.10}, one has $V(\hat K^k_j)\approx V(K^k_j)$. The fact that $p\geq 2$ and the containment $K^k_{j}\subseteq \hat K^k_{j}$ gives the inequality
$\left(\nu(\hat K^k_{j})\right)^{2-p}\leq\left(\nu( K^k_{j})\right)^{2-p}$. This inequality yields:
\begin{align}
\left(V(\hat K^k_{j})\right)^{p-1}\left(\nu(\hat K^k_{j})\right)^{2-p}\lesssim \left(V( K^k_{j})\right)^{p-1}\left(\nu( K^k_{j})\right)^{2-p}.
\end{align}
By H\"older's inequality, $$V(K^k_{j})\leq \left(\nu(K^k_{j})\right)^{\frac{1}{p^\prime}}\left(\sigma({K^k_{j}})\right)^{\frac{1}{p}}.$$
Therefore,
\begin{equation}
 \left(V( K^k_{j})\right)^{p-1}\left(\nu( K^k_{j})\right)^{2-p}\leq \left(\nu( K^k_{j})\right)^{\frac{1}{p}}\left(\sigma({K^k_{j}})\right)^{\frac{1}{p^\prime}}.
\end{equation} 
Substituting these inequalities into the last line of (\ref{5.6}), we obtain
\begin{align*}
&\sum_{\hat K^k_j\in \mathcal T_{l}}\langle  f\rangle^{\nu dV}_{\hat K^k_{j}}\langle g\rangle^{\sigma dV}_{\hat K^k_{j}}\left(V(\hat K^k_{j})\right)^{p-1}\left(\nu(\hat K^k_{j})\right)^{2-p}\lesssim\mathcal \sum_{\hat K^k_j\in \mathcal T_{l}}\langle  f\rangle^{\nu dV}_{\hat K^k_{j}}\langle g\rangle^{\sigma dV}_{\hat K^k_{j}}\left(\nu( K^k_{j})\right)^{\frac{1}{p}}\left(\sigma({K^k_{j}})\right)^{\frac{1}{p^\prime}}.
\end{align*}
Applying H\"older's inequality again the sum above gives
\begin{align}\label{5.80}
&\sum_{\hat K^k_j\in \mathcal T_{l}}\langle  f\rangle^{\nu dV}_{\hat K^k_{j}}\langle g\rangle^{\sigma dV}_{\hat K^k_{j}}\left(\nu( K^k_{j})\right)^{\frac{1}{p}}\left(\sigma({K^k_{j}})\right)^{\frac{1}{p^\prime}}\nonumber
\\\leq&\left(\sum_{\hat K^k_j\in \mathcal T_{l}}\left(\langle  f\rangle^{\nu dV}_{\hat K^k_{j}}\right)^{p}\nu(K^k_{j})\right)^{\frac{1}{p}}\left(\sum_{\hat K^k_j\in \mathcal T_{l}}\left(\langle g\rangle^{\sigma dV}_{K^k_{j}}\right)^{p^\prime}\sigma({K^k_{j}})\right)^{\frac{1}{p^\prime}}.
\end{align}
By the disjointness of $K^k_{j}$ and Lemma \ref{lem3.12}, we have 
\begin{equation}\label{5.8}
\sum_{\hat K^k_j\in \mathcal T_{l}}\left(\langle  f\rangle^{\nu dV}_{\hat K^k_{j}}\right)^{p}\nu( K^k_{j})\leq \int_{\Omega} (\mathcal M_{\mathcal T_{l},\nu}f)^p\nu dV\leq (p^\prime)^{p}\|f\|^{p}_{L^p(\Omega,\nu)}.
\end{equation}
Similarly, we also have
\begin{equation}\label{5.9}
\sum_{\hat K^k_j\in \mathcal T_{l}}\left(\langle  g\rangle^{\sigma dV}_{\hat K^k_{j}}\right)^{p^\prime}\sigma( K^k_{j})\leq \int_{\Omega} (\mathcal M_{\mathcal T_{l},\sigma}g)^{p^\prime}\sigma dV\leq (p)^{p^\prime}\|g\|^{p^\prime}_{L^{p^\prime}(\Omega,\sigma)}.
\end{equation}
Substituting (\ref{5.8}) and (\ref{5.9}) back into (\ref{5.80}) and (\ref{5.5}), we finally obtain
\begin{equation*}
\left\langle Q^+_{l,\nu} f, g\sigma\right\rangle\lesssim pp^\prime \sup_{1\leq l\leq N}\sup_{\hat K^k_j\in \mathcal T_{l}}\langle\sigma \rangle^{dV}_{\hat K^k_j}\left(\langle \nu\rangle^{dV}_{\hat K^k_j}\right)^{p-1}\|f\|_{L^p(\Omega,\nu )} \|g\|_{L^{p^\prime}(\Omega,\sigma)}.
\end{equation*}
Therefore $\sum_{l=1}^N\|Q^+_{l,\nu}\|_{L^p(\Omega, \sigma)}\lesssim pp^\prime \sup_{1\leq l\leq N}\sup_{\hat K^k_j\in \mathcal T_{l}}\langle\sigma \rangle^{dV}_{\hat K^k_j}\left(\langle \nu\rangle^{dV}_{\hat K^k_j}\right)^{p-1}$.

Now we turn to the case $1<p<2$ and show that for all $f\in L^p(\Omega,\nu)$ and $g\in L^{p^\prime}(\Omega,\sigma)$
\begin{equation*}
\left \langle Q^+_{l,\nu} f, g\sigma\right\rangle\lesssim\left(\sup_{1\leq l\leq N}\sup_{\hat K^k_j\in \mathcal T_{l}}\langle\sigma \rangle^{dV}_{\hat K^k_j}\left(\langle \nu\rangle^{dV}_{\hat K^k_j}\right)^{p-1}\right)^{\frac{1}{p-1}}\|f\|_{L^p(\Omega,\nu)} \|g\|_{L^{p^\prime}(\Omega,\sigma)}.
\end{equation*}
By the definition of $Q^+_{l,\nu}$,
\begin{align}\label{5.11}
\left \langle Q^+_{l,\nu} f, g\sigma\right\rangle&=\left\langle\sum_{\hat K^k_j\in \mathcal T_{l}}1_{\hat K^k_{j}}(w)\langle  f\nu\rangle^{dV}_{\hat K^k_{j}},g\sigma\right\rangle\nonumber\\&=\sum_{\hat K^k_j\in \mathcal T_{l}}\langle  f\nu\rangle^{dV}_{\hat K^k_{j}}\langle g\sigma\rangle^{dV}_{\hat K^k_{j}}V(\hat K^k_{j})\nonumber
\\&=\sum_{\hat K^k_j\in \mathcal T_{l}}\left\langle 1_{\hat K^k_{j}}(w)\langle  g\sigma\rangle^{dV}_{\hat K^k_{j}},f\nu\right\rangle
\nonumber\\&=\left\langle Q^+_{l,\sigma}(g),f\nu\right\rangle.
\end{align}Since $1<p<2$, $p^\prime>2$. Then, replacing $p$ by $p^\prime$, interchanging $\sigma$ and $\nu$, and adopting the same argument for the case $p\geq2$ yields that
\begin{align*}
\|Q^+_{l,\sigma} \|_{L^{p^\prime}(\Omega,\nu )}&\lesssim p^\prime p\sup_{1\leq l\leq N}\sup_{\hat K^k_j\in \mathcal T_{l}}\langle\nu \rangle^{dV}_{\hat K^k_j}\left(\langle \sigma\rangle^{dV}_{\hat K^k_j}\right)^{p^\prime-1}
\nonumber\\&=pp^\prime\left(\sup_{1\leq l\leq N}\sup_{\hat K^k_j\in \mathcal T_{l}}\langle\sigma \rangle^{dV}_{\hat K^k_j}\left(\langle \nu\rangle^{dV}_{\hat K^k_j}\right)^{p-1}\right)^{\frac{1}{p-1}}.
\end{align*}
Thus we have $$\left \langle Q^+_{l,\nu} f, g\sigma\right\rangle\lesssim pp^\prime\left(\sup_{1\leq l\leq N}\sup_{\hat K^k_j\in \mathcal T_{l}}\langle\sigma \rangle^{dV}_{\hat K^k_j}\left(\langle \nu\rangle^{dV}_{\hat K^k_j}\right)^{p-1}\right)^\frac{1}{p-1}\|g\|_{L^{p^\prime}(\Omega,\sigma )}\|f\|_{L^p(\Omega,\nu)},$$ and $$\|Q^+_{l,\nu}:L^p(\Omega,\nu)\to L^p(\Omega,\sigma)\|\lesssim pp^\prime\left(\sup_{1\leq l\leq N}\sup_{\hat K^k_j\in \mathcal T_{l}}\langle\sigma \rangle^{dV}_{\hat K^k_j}\left(\langle \nu\rangle^{dV}_{\hat K^k_j}\right)^{p-1}\right)^{\frac{1}{p-1}}.$$
Combining the results for $Q^+_{l,\nu}$ with $l\neq 0$ for $1<p< 2$ and $p\geq2$ and inequality (\ref{5.50}) for $Q^+_{0,\nu}$ gives the estimate in Theorem \ref{t:main}: $$\|P^+\|_{L^p(\Omega,\sigma)}\lesssim [\sigma]_p.$$
\section{A sharp example}
In this section, we provide an example to show that the estimate in Theorem \ref{t:main} is sharp. Our example is for the case $1<p\leq2$. The case $p>2$ follows by a duality argument.
The idea is similar to the ones in \cite{Pott,Rahm}.  Since $\Omega$ is a pseudoconvex domain of finite type, Kerzman's Theorem \cite{Kerzman,Boas} implies that the kernel function $K_\Omega$ extends to a $C^\infty$ function away from a neighborhood of the boundary diagonal of $\Omega\times\Omega$.  Let $w_\circ\in \Omega$ be a point that is away from the set  $N_{\epsilon_0}(\mathbf b\Omega)$ and satisfies $K_\Omega(w_\circ;\bar w_\circ)=2C_1>0$ for some constant $C_1$. Then the maximum principle implies that $\{z\in\mathbf b\Omega:|K_\Omega(z;\bar w_\circ)|>C_1\}$ is a non-empty open subset of $\mathbf b\Omega$. We claim that there exists a strictly pseudoconvex point $z_\circ$ in the set $\{z\in\mathbf b\Omega:|K_\Omega(z;\bar w_\circ)|>C_1\}$.  Let $p$ be a point in $\{z\in\mathbf b\Omega:|K_\Omega(z;\bar w_\circ)|>C_1\}$. Since $p$ is a point of finite 1-type in the sense of D'Angelo, the determinant of the Levi form does not vanish identically near point $p$. Thus the determinant of the Levi form is strictly positive at some point in every neighborhood of $p$, i.e., there is a sequence of strictly pseudoconvex points converging to $p$.  Since $\{z\in\mathbf b\Omega:|K_\Omega(z;\bar w_\circ)|>C_1\}$ is a neighborhood of $p$,
 there exists a strictly pseudoconvex  point $z_\circ$ such that \begin{equation}\label{7.10}|K_{\Omega}(z_\circ;\bar w_\circ)|> C_1.\end{equation}  There are several possible proofs in the literature for the existence of the nonvanishing point of the determinant of the Levi form near a point of finite type. See for example \cite{Nicoara} or the forthcoming thesis of Fassina \cite{Fassina}. Nevertheless, we choose the strictly pseudoconvex point $z_\circ$ above only for the simplicity of the construction of our example and it is not required. See Remark \ref{re 6.1} below.  By Kerzman's Theorem again,  there exists a small constant $\delta_0$ such that for any pair of points $(z,w)\in B^\#(z_\circ,\delta_0)\times \{w\in\Omega:\operatorname{dist}(w,w_\circ)<\delta_0\}$,
\[|K_\Omega(z, \bar w)-K_{\Omega}(z_\circ;\bar w_\circ)|\leq C_1/10.\]
Thus for $(z,w)\in B^\#(z_\circ,\delta_0)\times \{w\in\Omega:\operatorname{dist}(w,w_\circ)<\delta_0\}$, one has
\[|K_\Omega(z, \bar w)|\approx C_1.\]
 Moreover, elementary geometric reasoning yields that \begin{align}\label{7.11}\arg\{K_{\Omega}(z;\bar w),K_{\Omega}(z_\circ,w_\circ)\}\in [-\sin^{-1}(1/10),\sin^{-1}(1/10)].\end{align} From now on, we let $\delta_0$ be a  fixed constant.
For $z\in\Omega$, we set
\begin{align}\label{6.3}
h(z)=\inf\{\delta>0:z\in B^\#(z_\circ,\delta)\}, \text{ and }\;\;l(z)=\operatorname{dist}(z,w_\circ).
\end{align}
Let $1<p\leq 2$. Let $s$ be a positive constant that is sufficiently close to $0$. We define the weight function $\sigma$ on $\Omega$ to be
\begin{align}
\sigma(w)=\frac{(h(w))^{(p-1)(2+2n-2s)}}{(l(w))^{2n-2s}}.
\end{align}
We claim that the constant $ [\sigma]_p\approx s^{-1}$. 

First, we consider the average of $\sigma$ and $\sigma^{\frac{1}{1-p}}$ over the tent $B^\#(z,\delta)$. Note that \[\sigma^{\frac{1}{1-p}}(w)=\frac{(h(w))^{(2s-2n-2)}}{(l(w))^{(2s-2n)/(p-1)}}.\] If the tent $B^\#(z,\delta)$ does not intersect $B^\#(z_\circ,\delta)$, then for any $w\in B^\#(z,\delta)$ we have  $h(w)\approx x+\delta$ and $l(w)\approx 1$ with $x\gtrsim \delta$. Thus we have 
\begin{equation}\label{7.4}\langle \sigma\rangle^{dV}_{B^\#(z,\delta)}\left(\langle \sigma^{\frac{1}{1-p}}\rangle^{dV}_{B^\#(z,\delta)}\right)^{p-1}\approx {(x+\delta)^{(p-1)(2+2n-2s)}}\left({(x+\delta)^{(2s-2n-2)}}\right)^{p-1}= 1.\end{equation} If $B^\#(z,\delta)$ intersects $B^\#(z_\circ,\delta)$, then there exists a constant $C$ so that $B^\#(z_\circ,C\delta)$ contains $B^\#(z,\delta)$ with $|B^\#(z_\circ,C\delta)|\approx |B^\#(z,\delta)|$ by the doubling property of the ball $B^\#$.
Hence 
\[\langle \sigma\rangle^{dV}_{B^\#(z,\delta)}\left(\langle \sigma^{\frac{1}{1-p}}\rangle^{dV}_{B^\#(z,\delta)}\right)^{p-1}\lesssim \langle \sigma\rangle^{dV}_{B^\#(z_\circ,C\delta)}\left(\langle \sigma^{\frac{1}{1-p}}\rangle^{dV}_{B^\#(z_\circ,C\delta)}\right)^{p-1}.\]
Since $w_\circ$ is away from  the set $N_{\epsilon_0}(\mathbf b\Omega)$ and hence is away from any tents, $l(w)\approx 1$ for any $w\in B^\#(z_\circ,C\delta)$. It follows that
\[\langle \sigma\rangle^{dV}_{B^\#(z_\circ,C\delta)}\approx\int_{B^\#(z_\circ,C\delta)}h(w)^{(p-1)(2n+2-2s)}dV(w)(V(B^\#(z,\delta)))^{-1}.\] 
Recall that  $z_\circ$ is a strictly pseudoconvex point. There exist special holomorphic coordinates $(\zeta_1,\dots,\zeta_n)$ in a neighborhood of $z_\circ$ as in Theorem \ref{thm4.2} so that $z_\circ=(0,\dots, 0)$ and the tent
\[D(z_\circ,\delta):=\{w=(\zeta_1,\dots,\zeta_n)\in \Omega:|\zeta_n|<\delta^2, |\zeta_j|<\delta, j=1,\dots, n-1\},\] 
is equivalent to $B^\#(z_\circ,C\delta)$ in the sense that there exist constants $c_1$ and $c_2$ so that \[D(z_\circ,c_1\delta)\subseteq B^\#(z_\circ,C\delta)\subseteq D(z_\circ,c_2\delta).\]
Moreover, $h(w)\approx (|\zeta_1|^2+\cdots+|\zeta_{n-1}|^2+|\zeta_n|)^{\frac{1}{2}}$. Therefore 
\begin{equation}
\label{6.5}
\langle \sigma\rangle^{dV}_{B^\#(z_\circ,C\delta)}\approx\int_{B^\#(z_\circ,C\delta)}h(w)^{(p-1)(2n+2-2s)}dV(w)(V(B^\#(z,\delta)))^{-1}\approx\delta^{(p-1)(2n+2-2s)}.\end{equation}
Similarly,
\begin{equation}\label{6.6}\langle \sigma^{\frac{1}{1-p}}\rangle^{dV}_{B^\#(z_\circ,C\delta)}\approx\int_{B^\#(z_\circ,C\delta)}h(w)^{2s-2n-2}dV(w)(V(B^\#(z,\delta)))^{-1}\approx s^{-1}\delta^{(2s-2n-2)},\end{equation}
where $s^{-1}$ comes from the power rule $\int_0^a t^{s-1}dt={a^s}/{s}$.
Combining these inequalities yields
\begin{equation}\label{7.5}\langle \sigma\rangle^{dV}_{B^\#(z,\delta)}\left(\langle \sigma^{\frac{1}{1-p}}\rangle^{dV}_{B^\#(z,\delta)}\right)^{p-1}\lesssim\langle \sigma\rangle^{dV}_{B^\#(z_\circ,C\delta)}(\langle \sigma^{\frac{1}{1-p}}\rangle^{dV}_{B^\#(z_\circ,C\delta)})^{p-1}\approx s^{-(p-1)}.\end{equation}

 Now we turn to the average of $\sigma$ and $\sigma^{\frac{1}{1-p}}$ over the entire domain $\Omega$. Note that $$\langle \sigma\rangle^{dV}_{\Omega}\approx\int_\Omega l(w)^{2s-2n}d V(w).$$ A computation using polar coordinates yields that
$\langle \sigma\rangle^{dV}_{\Omega} \approx s^{-1}$. Also, 
\begin{align*}\langle \sigma^{\frac{1}{1-p}}\rangle^{dV}_{\Omega}&\approx\int_\Omega h(w)^{2s-2n-2}d V(w)\\&=\int_{B^\#(z_\circ,\delta_0)} h(w)^{2s-2n-2}d V(w)+\int_{\Omega\backslash B^\#(z_\circ,\delta_0)} h(w)^{2s-2n-2}d V(w)\\&\approx \delta_0^{2s-2n-2}s^{-1}+ \delta_0^{2s-2n-2}\approx s^{-1},\end{align*}
 where the third approximation sign follows by $s$ being sufficiently small and $\delta_0$ being a fixed constant. Thus 
\begin{equation}\label{7.6}\langle \sigma\rangle^{dV}_{\Omega}\left(\langle \sigma^{\frac{1}{1-p}}\rangle^{dV}_{\Omega}\right)^{p-1}\approx s^{-1}(s^{-1})^{p-1}\approx s^{-p}.\end{equation}
This estimate together with inequalities (\ref{7.4}) and (\ref{7.5}) yields that $[\sigma]_p\approx s^{-1}$. 

Now we consider the function 
\[f(w)=\sigma^{\frac{1}{1-p}}(w)1_{B^\#(z_\circ,\delta_0)},\]
where $\delta_0$ is the same fixed constant so that (\ref{7.11}) holds.
Since $z_\circ$ is a point away from $w_\circ$, $l(w)\approx 1$ for $w\in B^\#(z_\circ,\delta_0)$. Thus
\[\|f\|^p_{L^p(\Omega,\sigma)}=\langle \sigma^{\frac{1}{1-p}}\rangle^{dV}_{B^\#(z_\circ,\delta_0)}V(B^\#(z_\circ,\delta_0))\approx s^{-1}.\]
When $z\in \{w\in\Omega:\operatorname{dist}(w,w_\circ)<\delta_0\}$, (\ref{7.10}) and (\ref{7.11}) imply that
\begin{align}\label{6.9}
|P(f)(z)|&=\left|\int_{B^\#(z_\circ,\delta_0)}K_\Omega(z;\bar w)f(w)dV(w)\right|\nonumber\\&\approx \int_{B^\#(z_\circ,\delta_0)}|K_\Omega(z_0;\bar w_0)|f(w)dV(w)\nonumber\\&\approx \int_{B^\#(z_\circ,\delta_0)}f(w)dV(w)=\langle\sigma^{\frac{1}{1-p}}\rangle^{dV}_{B^\#(z_\circ,\delta_0)}V(B^\#(z_\circ,\delta_0))\approx s^{-1}.
\end{align}
By (\ref{6.9}) and the fact that $\delta_0\approx 1$, we obtain the desired estimate:
\begin{align}
\|P(f)\|^p_{L^p(\Omega,\sigma)}&=\int_\Omega|P(f)(z)|^p\sigma(z)dV(z)\nonumber\\&\geq\int_{\{z\in\Omega:\operatorname{dist}(z,w_\circ)<\delta_0\}}|P(f)(z)|^p\sigma(z)dV(z)\nonumber\\
&\gtrsim s^{-p}\int_{\{z\in\Omega:\operatorname{dist}(z,w_\circ)<\delta_0\}} (\operatorname{dist}(z,w_\circ))^{2s-2n}dV(z)\nonumber\\&\approx s^{-p}s^{-1}\approx \mathcal ([\sigma]_p)^p\|f\|^p_{L^p(\Omega,\sigma)}.
\end{align}
\begin{rmk}In the particular case of the unit ball $\mathbb{B}_n$ we can make our example more explicit:  the weight $\sigma(w)=|w-z_\circ|^{(p-1)(2+2n-2s)}/|w|^{2n-2s}$ with $z_\circ=(1,0,\dots,0)$ and the test function $f(w)=\sigma^{\frac{1}{1-p}}(w)1_{B^\#(z_\circ,1/2)}(w)$.  One can compute explicitly in this case that $\sigma$ is in the $\mathcal B_p$ class.
	We further remark that in \cite{Rahm}, the authors produce an upper and lower bound in terms of a Bekoll\'{e}-Bonami condition that doesn't utilize information about the large tents.  The upper bound they produce is correct, however the claimed sharpness of the Bekoll\'e-Bonami condition without testing the large tents is not quite correct.  The example they construct does appropriately capture the behavior of small tents, but fails to do so in the case of large tents and this characteristic fails to capture the sharpness.  It is for this reason that we have had to modify the definition of the Bekoll\'{e}-Bonami characteristic in Definition \ref{de3.4} to reflect the behavior of both large and small tents.\end{rmk}

\begin{rmk}\label{re 6.1}In this example, we require $z_\circ$ to be a strictly pseudoconvex point only for the simplicity of the construction of the weight $\sigma$ and the test function $f$, and the computation. For every $z\in\mathbf b\Omega$, the geometry of the tent $B^\#(z,\delta)$  is well understood. Thus $\sigma$ and  $f$ can be modified accordingly so that the estimate (\ref{7.5}) still holds true. 
	\end{rmk}

\begin{rmk}
	For a different example, we can  also choose $z_\circ$ to be a point in $\Omega$ that is away from both $N_{\epsilon_0}(\mathbf b\Omega)$  and $w_\circ$, and change $h(w)$ in (\ref{6.3}) to be $(\operatorname{dist}(z_\circ, w))^{(p-1)(2n-2s)}$. The average of $\sigma$ and $\sigma^{\frac{1}{1-p}}$ over tents is controlled by a constant since all tents are away from points $z_\circ$ and $w_\circ$. Moreover,  $\langle\sigma\rangle^{dV}_\Omega(\langle\sigma^{\frac{1}{p-1}}\rangle^{dV}_\Omega)^{p-1}\approx s^{-p}$ by a computation using polar coordinates.
Thus $[\sigma]_p\approx s^{-1}$ and a similar argument yields the sharpness of the bound in Theorem \ref{t:main}. We did not adopt this example since  it does not reflect the connection between the weighted norm of the projection and  the average of $\sigma$ and $\sigma^{\frac{1}{1-p}}$ over small tents.
 \end{rmk}
\begin{rmk}When the weight $\sigma\equiv1$, the constant $[\sigma]_p\approx pp^\prime$. Theorem \ref{t:main} then gives an estimate for the $L^p$ norm of the Bergman projection: \[\|P\|_{L^p(\Omega)}\lesssim pp^\prime.\]	For the strictly pseudoconvex case, such an estimate was obtained and proven to be sharp by \v{C}u\v{c}kovi\'{c} \cite{Cuckovic17}.  Therefore, the constant $pp^\prime$ in  $[\sigma]_p$ is necessary.\end{rmk}
\section{Proof of Theorem \ref{t:main1}}
We first show that the lower bound (\ref{1.2}) in Theorem \ref{t:main1} with the assumption that $\Omega$ is bounded, smooth, and strictly pseudoconvex.

We begin by recalling the following two lemmas from \cite{HWW}.
\begin{lem}\label{Lem7.0}
	Let $\Omega$ be a smooth, bounded, strictly pseudoconvex domain. If the Bergman projection $P$ is bounded on the weighted space $L^p(\Omega,\sigma)$, then the weight $\sigma$ and its dual weight $\nu=\sigma^{\frac{1}{1-p}}$ are integrable on $\Omega$.
\end{lem}
	
\begin{lem}\label{Lem7.1}
	Let $\Omega$ be a smooth, bounded, strictly pseudoconvex domain. 	Let $\delta$ be a small constant. For a boundary point $z_1$, let $B^\#(z_1,\delta)$ be a tent defined as in Definition \ref{de3.3}. Then there exists a tent $B^\#(z_2,\delta)$ with $d(B(z_1,\delta),B(z_2,\delta))\approx \delta$ so that if $f\geq 0$ is a function supported in $B^\#(z_i,\delta)$ and $z\in B^\#(z_j,\delta)$ with $i\neq j$ and $i,j\in \{1,2\}$, then we have
	$$|P(f)(z)|\gtrsim \langle f\rangle^{dV}_{B^\#(z_i,\delta)}.$$
\end{lem}

Recall that $\nu=\sigma^{1/(1-p)}$. By (\ref{5.1}), 
\begin{equation*}
\|P\|_{L^p(\Omega,\sigma dV)}= \|PM_\nu:L^p(\Omega,\nu dV)\to L^p(\Omega,\sigma dV)\|.
\end{equation*}
It suffices to show that $$\sup_{\epsilon_0>\delta>0, z\in \mathbf b\Omega}\langle\sigma \rangle^{dV}_{B^\#(z,\delta)}\left(\langle \nu\rangle^{dV}_{B^\#(z,\delta)}\right)^{p-1}\lesssim \|PM_\nu:L^p(\Omega,\nu dV)\to L^p(\Omega,\sigma dV)\|^{2p}.$$ For simplicity, we set $\mathcal A:=\|PM_\nu:L^p(\Omega,\nu dV)\to L^p(\Omega,\sigma dV)\|$.  If $\mathcal A<\infty$, then we have a weak-type $(p,p)$ estimate:
\begin{equation}\label{3.19}
\sigma\{w\in\Omega:|PM_\nu f(w)|>\lambda\}\lesssim\frac{\mathcal A^{p}}{\lambda^p}\|f\|^p_{L^p(\Omega,\nu dV)}.
\end{equation}
Let $\delta_0$ be a fixed constant so that Lemma \ref{Lem7.1} is true for all $\delta<\delta_0$. Set $f(w)=1_{B^\#(z_1,\delta)}(w)$. Lemma \ref{Lem7.1} implies that for any $z\in B^\#(z_2,\delta)$,
\begin{align}
|PM_\nu1_{B^\#(z_1,\delta)}(z)|=&\left|\int_{ B^\#(z_1,\delta)}K_\Omega(z;\bar w)\nu(w) dV(w)\right|> \langle \nu\rangle^{dV}_{B^\#(z_1,\delta)}.
\end{align} 
It follows that
\begin{equation}
 B^\#(z_2,\delta)\subseteq \{w\in\Omega:|PM_\nu f(w)|>\langle \nu\rangle^{dV}_{B^\#(z_1,\delta)}\}.
\end{equation}
By Lemma \ref{Lem7.0},  $\langle \nu\rangle^{dV}_{B^\#(z_1,\delta)}$ is finite. Then inequality (\ref{3.19}) implies 
\begin{equation}
\sigma(B^\#(z_2,\delta))\leq\mathcal A^p\left(\langle\nu\rangle_{B^\#(z_1,\delta)}^{dV}\right)^{-p}\nu(B^\#(z_1,\delta)),
\end{equation}
which is equivalent to $\langle\sigma\rangle_{B^\#(z_2,\delta)}^{dV}\left(\langle\nu\rangle_{B^\#(z_1,\delta)}^{dV}\right)^{p-1}\lesssim \mathcal A^p$. Since one can interchange the roles of $z_1$ and $z_2$ in Lemma \ref{Lem7.1}, it follows that $$\langle\sigma\rangle_{B^\#(z_1,\delta)}^{dV}\left(\langle\nu\rangle_{B^\#(z_2,\delta)}^{dV}\right)^{p-1}\lesssim \mathcal A^p.$$  Combining these two inequalities, we have
\begin{equation}\label{3.25}
\langle\sigma\rangle_{B^\#(z_1,\delta)}^{dV}\left(\langle\nu\rangle_{B^\#(z_2,\delta)}^{dV}\right)^{p-1}\langle\sigma\rangle_{B^\#(z_2,\delta)}^{dV}\left(\langle\nu\rangle_{B^\#(z_1,\delta)}^{dV}\right)^{p-1}\lesssim \mathcal A^{2p}.
\end{equation}
By H\"older's inequality,
\begin{equation}
V(B^\#(z_2,\delta))^{p}\leq \int_{B^\#(z_2,\delta)}\sigma dV\left(\int_{ B^\#(z_2,\delta)}\nu dV\right)^{p-1}.
\end{equation}
 Therefore $\langle\sigma\rangle_{B^\#(z_2,\delta)}^{dV}\left(\langle\nu\rangle_{B^\#(z_2,\delta)}^{dV}\right)^{p-1}\gtrsim 1$. Applying this to (\ref{3.25}) and taking the supremum of the left side of (\ref{3.25}) for all tents $B^\#(z_1,\delta)$ where $\delta<\delta_0$ yields
\begin{equation}\label{3.27}
\sup_{\substack{\delta<\delta_0,\\z_1\in \mathbf b\Omega}}\langle \sigma\rangle_{B^\#(z_1,\delta)}^{dV}\left(\langle\nu\rangle_{B^\#(z_1,\delta)}^{dV}\right)^{p-1}\lesssim \mathcal A^{2p}.
\end{equation}
Since the constant $\epsilon_0$ in Lemma \ref{lem3.2} can be chosen to be $\delta_0$, inequality (\ref{1.2}) is proved.

Now we turn to prove (\ref{1.3}) and assume in addition that $\Omega$ is Reinhardt. Since inequality (\ref{3.27}) still holds true, it suffices to show
\begin{equation}\label{7.8}
\langle \sigma\rangle_{\Omega}^{dV}\left(\langle\nu\rangle_{\Omega}^{dV}\right)^{p-1}\lesssim \mathcal A^{2p}.
\end{equation}
Because $\Omega$ is Reinhardt,  the monomials form a complete orthogonal system for the Bergman space $A^2(\Omega)$. Thus the kernel function $K_\Omega$ has the following series expression:
\begin{align}
K_\Omega(z;\bar w)=\sum_{\alpha\in \mathbb N^n} \frac{z^\alpha \bar w^\alpha}{\|z^\alpha\|^2_{L^2(\Omega)}}.
\end{align}
This implies that $K_\Omega(z;0)=\|1\|^{-2}_{L^2(\Omega)}$ for any $z\in\Omega$. By either the asymptotic expansion of $K_\Omega$ \cite{Fefferman, Monvel} or Kerzman's Theorem \cite{Kerzman}, we can find  a precompact neighborhood $U$ of the origin such that  for any $z\in \Omega$ and $w\in U$,
\begin{align}
|K_\Omega(z;\bar w)|\approx 1 \;\;\;\;\text{ and }\;\;\;\;\arg\{K_\Omega(z;\bar w), K_\Omega(z;0)\}\in [-1/4,1/4].
\end{align}
Let  $f(w)=1_U\left(w\right)$. Then for any  $z\in \Omega$,
\begin{align*}&\left|PM_\nu(f)\left(z\right)\right|=\left|\int_{U}K_\Omega(z;\bar w)\nu dV(w)\right|> c\|f\|_{L^1(\Omega,\nu dV)},\end{align*} for some constant $c$.
Therefore,
\begin{equation*}
\Omega\subseteq\left\{z\in \Omega:|PM_\nu(f)(z)|>c\|f\|_{L^1(\Omega,\nu dV)}\right\}.
\end{equation*}
Applying this containment and  the fact that $\|f\|_{L^1(\Omega,\nu dV)}=\|f\|^p_{L^p(\Omega,\nu dV)}$ to (\ref{3.19})  yields
\begin{equation}\label{7.111}
\sigma(\Omega)\leq\frac{\mathcal A^p}{c^p\|f\|^p_{L^1(\Omega,\nu dV)}}\|f\|^p_{L^p(\Omega,\nu dV)}\leq \frac{\mathcal A^p}{c^p\|f\|^{p-1}_{L^1(\Omega,\nu dV)}}<\infty.
\end{equation}
Thus \begin{align}\label{7.12}\langle\sigma\rangle^{dV}_\Omega\left(\langle\nu\rangle^{dV}_U\right)^{p-1}\lesssim \mathcal A^p.\end{align} Interchanging the role of $z$ and $w$ in the argument above, we also have
\begin{equation*}
U\subseteq\left\{w\in \Omega:|PM_\nu(1)(w)|>c\|1\|_{L^1(\Omega,\nu dV)}\right\},
\end{equation*}
and 
\begin{equation}\label{7.112}
\sigma(U)\leq\frac{\mathcal A^p}{c^p\|1\|^p_{L^1(\Omega,\nu dV)}}\|1\|^p_{L^p(\Omega,\nu dV)}\leq \frac{\mathcal A^p}{c^p\|1\|^{p-1}_{L^1(\Omega,\nu dV)}}<\infty.
\end{equation}
Thus 
\begin{align}\label{7.13}\langle\sigma\rangle^{dV}_U\left(\langle\nu\rangle^{dV}_\Omega\right)^{p-1}\lesssim \mathcal A^p.\end{align}
Combining (\ref{7.12}), (\ref{7.13}) and using the fact that $$\langle\sigma\rangle^{dV}_U\left(\langle\nu\rangle^{dV}_U\right)^{p-1}\geq 1,$$ we obtain the desired estimate: \begin{align}\label{7.15}\langle\sigma\rangle^{dV}_\Omega\left(\langle\nu\rangle^{dV}_\Omega\right)^{p-1}\lesssim\langle\sigma\rangle^{dV}_\Omega\left(\langle\nu\rangle^{dV}_\Omega\right)^{p-1}\langle\sigma\rangle^{dV}_U\left(\langle\nu\rangle^{dV}_U\right)^{p-1}\lesssim \mathcal A^{2p}.\end{align}
Estimates (\ref{7.15}) and (\ref{3.27}) then give (\ref{1.3}). The proof is complete.

\section{An application to the weak $L^1$ estimate}
 In \cite{McNeal3}, the weak-type $(1,1)$ boundedness of the Bergman projection on  simple domains was obtained using a Calderon-Zygmund type decomposition. In this section, we  use Theorem \ref{4.6} to provide an alternative approach to establish the weak-type bound for the Bergman projection. We follow the argument in \cite{CACPO17} since we have a ``sparse domination" for the Bergman projection.
\begin{thm}
There exists a constant $C>0$ so that for all $f \in L^1(\Omega),$
$$\sup_{\lambda} \lambda V(\{z: |Pf(z)|>\lambda\})< C \|f\|_{L^{1}(\Omega)}.$$

\begin{proof}
By a well-known equivalence of weak-type norms (see for example \cite{Grafakos}), it suffices to show
\begin{equation}\label{8.1} \sup_{\substack {f_1 \\ \|f_1\|_{L^1(\Omega)}=1}} \sup_{G \subset \Omega} \inf_{\substack{G' \subset G \\ V(G)< 2 V(G')}} \sup_{\substack{f_2 \\ |f_2| \leq 1_{G'}}} |\langle P f_1,f_2 \rangle| < \infty. \end{equation}
In light of Theorem \ref{4.6}, we may replace $P$ by $Q^+_{\ell_0,1}$ (using our previous notation) for some fixed $\ell_0$ with $1\leq \ell \leq N$. As in Definition \ref{de3.12}, we consider the (now unweighted) dyadic maximal function $\mathcal{M}_{\mathcal T_{\ell_0},1}$. For convenience in what follows, we will simply write $\mathcal M_{\mathcal T_{\ell_0}}$. By Lemma \ref{lem3.12}, we know this operator is of weak-type $(1,1)$. Fix $f_1$ with norm $1$, $G \subset \Omega$ and constants $C_1$, $C_2$ to be chosen later. Define sets 
$$H=\{z \in \Omega: \mathcal M_{\mathcal T_{\ell_0}}f_1(z)>C_1 V(G)^{-1}\}$$
and
$$\tilde{H}= \bigcup_{\hat{K}_j^k \in \mathcal{K}} \hat{K}_j^k$$
where
$$\mathcal{K}=\left\{\text{maximal tents $\hat{K}_j^k$ in $\mathcal{T}_{\ell_0}$ so $V(\hat{K}_j^k \cap H)>C_2 V(\hat{K}_j^k)$}\right\}.$$
Note if $C_1$ is chosen sufficiently large relative to $C_2^{-1}$, the weak-type estimate of $\mathcal M_{\mathcal T_{\ell_0}}$ implies
\begin{align*}
V(\tilde{H})& =V\left( \bigcup_{\hat{K}_j^k \in \mathcal{K}} \hat{K}_j^k \right)\\
& \leq \sum_{\hat{K}_j^k \in \mathcal{K}} C_2^{-1}V(\hat{K}_j^k \cap H)\\
& \leq C_2^{-1} V(H)\\
& \leq C_2^{-1} C_1^{-1} V(G) \|f_1\|_{L^1(\Omega)}\\
& \leq \frac{1}{2} V(G).
\end{align*}
It is then clear if we let $G'=G \setminus \tilde{H}$, then $V(G)<2 V(G')$, so $G'$ is a candidate set in the infimum in \eqref{8.1}.
If $z \in H^c$, then, by definition,
\begin{equation}\mathcal M_{\mathcal T_{\ell_0}} f_1(z) \leq C_1 V(G)^{-1}. \end{equation} \label{8.2}
Using the distribution function,
\begin{align}
\|\mathcal {M}_{\mathcal T_{\ell_0}} f_1\|^{2}_{L^{2}(H^c)}&=2\int_{0}^{ C_1 V(G)^{-1}} tV(\{z\in H^c:\mathcal {M}_{\mathcal T_{\ell_0}} f_1(z)>t\}) dt\nonumber\\&\leq 2\int_{0}^{ C_1 V(G)^{-1}}dt\|\mathcal {M}_{\mathcal T_{\ell_0}} \|_{L^{1,\infty}(H^c)}\|f_1\|_{L^1(\Omega)}\nonumber\\
& \lesssim C_1V(G)^{-1}\label{8.4}.
\end{align} 
Now let $|f_2 |\leq 1_G$ be fixed. We have
\begin{equation} |\langle Q_{\ell_0,1}^{+} f_1, f_2 \rangle|= \sum_{\hat{K}_j^k \in \mathcal T_{\ell_0}} V(\hat{K}_j^k) \langle f_1 \rangle_{\hat{K}_j^k} \langle f_2 \rangle_{\hat{K}_j^k}.  \label{8.5} \end{equation} 
Note that for $\hat{K}_j^k \in \mathcal T_{\ell_0}$, if $V(\hat{K}_j^k \cap H)> C_2 V(\hat{K}_j^k)$ then $\hat{K}_j^k \subset \tilde{H}$. But $f_2$ is supported on $G' \subset \tilde{H}^c$, so for such a tent $\langle f_2 \rangle_{\hat{K}_j^k}=0$. Thus, examining \eqref{8.5}, we may assume without loss of generality that if $\hat{K}_j^k \in \mathcal T_{\ell_0}$ then
 \begin{equation} V(\hat{K}_j^k \cap H)\leq C_2 V(\hat{K}_j^k). \label{8.6}\end{equation} 
Then note that \eqref{8.6} implies the following holds true for the kubes $K_j^k$, provided $C_2$ is chosen sufficiently small:
\begin{align*}
V(K_j^k \cap H^c) & = V(K_j^k)- V(K_j^k \cap H)\\
& \geq CV(\hat{K}_j^k)-V(\hat{K}_j^k \cap H)\\
& \gtrsim V(\hat{K}_j^k)\\
& \geq V(K_j^k)
\end{align*} 
where we let $C$ be the implicit constant in Lemma \ref{3.10}. Thus we have
\begin{equation} V(K_j^k) \lesssim V(K_j^k \cap H^c). \label{8.7} \end{equation}
Therefore, continuing from \eqref{8.5} and using \eqref{8.4} and \eqref{8.7}, we obtain

\begin{align*}
|\langle Q_{\ell_0,1}^+f_1,f_2 \rangle| & \lesssim 
\sum_{\hat{K}_j^k \in \mathcal T_{\ell_0}} V(K_j^k) \langle f_1 \rangle_{\hat{K}_j^k} \langle f_2 \rangle_{\hat{K}_j^k} \\
& \lesssim \sum_{\hat{K}_j^k \in \mathcal T_{\ell_0}} V(K_j^k \cap H^c) \langle f_1 \rangle_{\hat{K}_j^k} \langle f_2 \rangle_{\hat{K}_j^k}\\
& \leq \int_{H^c}(\mathcal M_{\mathcal T{\ell_0}}f_1)(\mathcal M_{\mathcal T{\ell_0}}f_2) \mathop{dV}\\
& \leq \|\mathcal M_{\mathcal T_{\ell_0}}f_1\|_{L^{2}(H^c)} \|\mathcal M_{\mathcal T_{\ell_0}}f_2\|_{L^2(\Omega)}\\
& \lesssim V(G)^{-\frac{1}{2}} \|f_2\|_{L^2(\Omega)}\\
& \leq V(G)^{-\frac{1}{2}} V(G)^{\frac{1}{2}}\\
& = 1,
\end{align*}
which establishes the result.

\end{proof}
\end{thm} 

\section{Directions for generalization}
\paragraph{1} 
The example in Section 6.1 showed the upper bound estimate in Theorem \ref{t:main} is sharp. It is not clear if the lower bound estimates given in Theorem \ref{t:main}, or in \cite{Pott} and \cite{Rahm} are sharp. It would be interesting to see what a sharp lower bound is in terms of the Bekoll\'e-Bonami type constant.
\vskip 5pt
\paragraph{2} Our lower bound estimate in Theorem \ref{t:main1} uses the asymptotic expansion of the Bergman kernel function and hence only works for bounded, smooth, strictly pseudoconvex domains. An interesting question would be whether similar lower bound estimates hold true for the Bergman projection when the domain is of finite type in $\mathbb C^2$, convex and of finite type in $\mathbb C^n$, or decoupled and of finite type in $\mathbb C^n$.
\vskip 5pt
\paragraph{3}  We focus on the weighted estimates for the Bergman projection for the simplicity of the computation. In \cite{Rahm}, Rahm, Tchoundja, and Wick obtained the weighted estimates for operators $S_{a,b}$ and $S^+_{a,b}$ defined by
\begin{align*}
S_{a,b}f(z)&:=(1-|z|^2)^a\int_{ \mathbb B_n }\frac{f(w)(1-|w|^2)^b}{(1-z\bar w)^{n+1+a+b}}dV(w);
\\S^+_{a,b}f(z)&:=(1-|z|^2)^a\int_{ \mathbb B_n }\frac{f(w)(1-|w|^2)^b}{|1-z\bar w|^{n+1+a+b}}dV(w),
\end{align*}
on the weighted space $L^p(\mathbb B_n,(1-|w|^2)^b\mu dV)$. Using the methods in this paper, it is possible to obtain weighted estimates for analogues of $S_{a,b}$ and $S^+_{a,b}$ in the settings we considered in this paper. 
\bibliographystyle{alpha}

\end{document}